\documentclass[a4paper]{article}

\usepackage{
amsmath,
amsthm,
amscd,
amssymb,
wasysym
}
\usepackage{comment}
\usepackage{tikz}
\usepackage{stmaryrd}
\usetikzlibrary{positioning,arrows,calc}
\usepackage{xspace}
\usepackage{caption}
\usepackage{subcaption}

\usepackage{graphicx}

\usepackage{url}

\theoremstyle{plain}
\newtheorem*{theorem*}{Theorem}
\newtheorem{lemma}{Lemma}
\newtheorem{definition}{Definition}
\newtheorem{corollary}{Corollary}

\newtheorem{theorem}{Theorem}
\newtheorem{conjecture}{Conjecture}
\newtheorem{question}{Question}

\newtheoremstyle{derp}
{3pt}
{3pt}
{}
{}
{\upshape}
{:}
{.5em}
{}
\theoremstyle{derp}
\newtheorem{example}{Example}

\newcommand{\Z}{\mathbb{Z}}
\newcommand{\C}{\mathcal{C}}
\newcommand{\power}{\mathcal{P}}
\newcommand{\N}{\mathbb{N}}

\newcommand{\B}{\mathcal{B}}

\newcommand{\ID}{\mathrm{id}}

\newcommand{\bla}{\mathrm{\textvisiblespace}}

\newcommand\xqed[1]{%
  \leavevmode\unskip\penalty9999 \hbox{}\nobreak\hfill
  \quad\hbox{#1}}
\newcommand\qee{\xqed{$\fullmoon$}}

\newcommand{\Alt}{\mathrm{Alt}}
\newcommand{\Sym}{\mathrm{Sym}}

\newcommand{\End}{\mathrm{End}}
\newcommand{\Aut}{\mathrm{Aut}}

\newcommand{\PAut}{\mathrm{PAut}}

\newcommand{\CPAut}{\mathrm{CPAut}}

\newcommand{\LP}{\mathrm{LP}}

\newcommand{\Homeo}{\mathrm{Homeo}}

\newcommand{\llb}{\llbracket}
\newcommand{\rrb}{\rrbracket}

\newcommand{\RFA}{\mathrm{RFA}}

\title{Conjugacy of reversible cellular automata and one-head machines}

\author{
Ville Salo \\
vosalo@utu.fi
}

\begin{document}
\maketitle

\begin{abstract}
We show that conjugacy of reversible cellular automata is undecidable, whether the conjugacy is to be performed by another reversible cellular automaton or by a general homeomorphism. This gives rise to a new family of finitely-generated groups with undecidable conjugacy problems, whose descriptions arguably do not involve any type of computation. For many automorphism groups of subshifts, as well as the group of asynchronous transducers and the homeomorphism group of the Cantor set, our result implies the existence of two elements such that every finitely-generated subgroup containing both has undecidable conjugacy problem. We say that conjugacy in these groups is eventually locally undecidable. We also prove that the Brin-Thompson group $2V$ and groups of reversible Turing machines have undecidable conjugacy problems, and show that the word problems of the automorphism group and the topological full group of every full shift are eventually locally co-NP-complete.
\end{abstract}

\section{Introduction and main results}

In 1911, Dehn \cite{De11} introduced three decision problems for finitely-presented (f.g.) groups: the word problem, the conjugacy problem, and the isomorphism problem. These problems are still considered some of the most fundamental decision problems for abstract groups. The problems naturally adapt to general countable groups, with respect to an enumeration of generators (for example a recursive presentation). The present paper is concerned with the first two problems, especially the conjugacy problem. 

The word problem can be undecidable for finitely-presented groups and solvable groups of small derived length \cite{No55,Bo57,Co86,Kh89}. For groups defined by a natural action, it tends to be decidable, usually almost by definition. Some examples where it is decidable include automatic groups (including hyperbolic and surface groups), one-relator groups (due to Magnus \cite{LySc15}), virtually polycyclic groups, finitely-presented simple groups, finitely-presented residually finite groups, 3-manifold groups \cite{Ep92}, (finitely-generated) linear groups over a field with decidable operations and equality (for example the algebraic numbers), automata groups, Thompson's groups $F,T,V$, Houghton's groups $H_n$. A large class of groups with decidable word problem, relevant to this paper, are f.g.\ subgroups of automorphism groups of subshifts with computable languages on groups with decidable word problem. This includes the groups of Turing machines from \cite{BaKaSa16} and topological full groups of topologically free subshifts under the same computability conditions.

For natural groups, the word problem is typically not just decidable, but efficiently so, and one approach to a more fine-grained study of the complexity is provided by complexity theory. The word problem is known to be efficiently solvable (polynomial time) in many cases, i.e.\ it is in the complexity class P. For example this is the case for automatic groups \cite{Ep92} and free-by-cyclic groups \cite{Sc08}. There are also completely natural groups where the word problem is ``harder'' than polynomial time:\footnote{The usual disclaimer is that while P $\subset$ co-NP $\subset$ PSPACE, it is famously not known whether these classes are distinct.} for the Brin-Thompson 2$V$ \cite{Bi20} and the group in \cite{Bi03}, the word problem is complete (under polynomial-time reductions) for the class co-NP of languages whose complements are polynomial-time verifiable, and the word problem in automata groups can be PSPACE-complete \cite{DARoWa17}.

The main results of this paper concern the conjugacy problem. The conjugacy problem has also been widely studied, and is known to be decidable for many natural classes of groups. Dehn showed this for surface groups, other (partially overlapping) classes where it is known to be decidable are polycyclic groups \cite{Re69} (see \cite{MaMyNiVa15} for a strong result on nilpotent groups), one-relator groups with torsion \cite{Ne68}, biautomatic groups \cite{Mo97}, hyperbolic groups (very efficiently \cite{EpHo06}), right-angled Artin groups and some natural subgroups of them \cite{CrGoWi09}, $3$-manifold groups \cite{Je12}, the Grigorchuk group \cite{LyMyUs10}, Houghton's groups \cite{AnBuMa15}, Thompson's $F$ \cite{GuSa97} and $V$ \cite{Hi74,Ol10}.

Conjugacy is also undecidable in general, even for finitely-presented groups and groups with decidable word problem. Its undecidability was first shown by Novikov \cite{No54}. Some of the simplest examples of groups with undecidable conjugacy problem are certain f.g.\ subgroups of $F_2 \times F_2$ with this property \cite{Mi68a}, free products with amalgamation $F_2 *_H F_2$ where $H \leq F_2$ is a suitably chosen finitely-generated subgroup \cite{Mi71}, and also $\Z^d \rtimes F_m$ \cite{SuVe12} for a suitable action of $F_m$ on $\Z^d$.

There are some key differences between the word problem and the conjugacy problem. First, while groups defined by abstract properties and presentations may have undecidable word problems, it is equally common to encounter a group in the wild since you have a natural action. It is hard to imagine the word problem being undecidable in the latter case -- if you have a natural action for your group by a natural action, intuitively you can decide the word problem by simply checking whether the given element acts nontrivially. The same is not a priori true for the conjugacy problem; even for groups arising from a natural action, there is no intuitive reason why there should be an algorithm for conjugacy. A second difference between these two problems is that, while word problems of subgroups with decidable word problem are also decidable, the conjugacy problem can become undecidable when passing to a subgroup.

In the case of subgroups of $F_2 \times F_2$ of \cite{Mi68a,LySc15} with undecidable conjugacy problems, it seems unlikely to the author that one would encounter the groups ``by accident'' by studying a natural action: Certainly the generators of these groups are defined by a very natural action (as they act by the restriction of the left-regular action on normal forms), but the precise choice of generators is somewhat complicated, and depends on the word problem for finitely-presented groups. The simplest known example of a finitely-presented group with undecidable word problem is presumably the one due to Collins, and is already quite complicated \cite{Co86}. Thus, in this example arguably undecidability arises because the choice of the subgroup is rather exotic and it is hard to figure out whether a conjugating element can be found in it. The supergroup $F_2 \times F_2$, like all right-angled Artin groups, of coure has decidable conjugacy problem. In the examples $F_2 *_H F_2$ and $\Z^d \rtimes F_m$ there is no obvious supergroup with decidable conjugacy, but it seems that in both cases the defining relations ultimately encode the \cite{Mi68a} example, and thus again it seems unlikely that faithful actions of these groups would arise by accident.

As mentioned, the first two of Dehn's problems make sense also for subgroups of a countable group with respect to a fixed enumeration of generators $(g_i)_{i \in \N}$. Our main result is about groups of reversible cellular automata (or \emph{RCA}) on a finite alphabet $A$. These are another name for automorphism groups of full shift dynamical systems $(A^\Z, \sigma)$, that is, the groups of shift-commuting homeomorphisms on Cantor spaces $A^\Z$, for finite alphabets $A$. These groups are recursively presented, and it is natural to present their elements by their ``local rules'', namely the inverse image of the clopen partition by the symbol at the origin, which constitutes a finite amount of data. 

Groups of reversible cellular automata have been studied at least in \cite{Ry72,BoLiRu88,Al88,KiRo90,Br93}, and more recently in at least \cite{Sa18d,FrScTa19,Sa17b,Sa18a,Sa19a}. Automorphism groups of other subshifts have been of much recent interest \cite{Ol13,SaTo15d,Sa14d,CyKr16a,CoYa14,CyKr16b,DoDuMaPe16,DoDuMaPe17,CyFrKrPe18,BaKaSa16,HaKrSc20}.
The word problem in finitely-generated groups of RCA was observed to be decidable in \cite{BoLiRu88}, and \cite{KiRo90} gives a better bound on the complexity (in particular implying that residually finite groups can have decidable word problem without being groups of RCA).

It is shown in \cite{KiRo90} that right-angled Artin groups act faithfully by RCA, so finitely-generated subgroups of $\Aut(A^\Z)$ can have undecidable conjugacy problem. However, these subgroups are somewhat complicated (even the embedding of $F_2 \times F_2$ is not entirely obvious). A perhaps more natural question is obtained by asking for conjugacy without restricting to a finitely-generated subgroup. There are two natural variants of the conjugacy problem, depending on whether the conjugacy is purely dynamical (realized by a homeomorphism) or whether we look for conjugacy in the group of RCA (which can dynamically be seen as conjugacy of the $\Z^2$-systems given by the joint action of the RCA and the shift). An analog of the latter question for non-reversible cellular automata was asked in 2017 \cite{Ep17}, and both types of conjugacy were shown undecidable in the non-reversible case in \cite{JaKa20}.

Cellular automata theory has a history of being a source of natural undecidable problems. Reversibility and surjectivity are undecidable for two-dimensional cellular automata \cite{Ka90}, and nilpotency\footnote{$f$ is \emph{nilpotent} if $\exists n \in \N, a \in A: f^n(x) = a^\Z$ for all $x \in A^\Z$} is undecidable for one-dimensional cellular automata \cite{AaLe74,Ka92}. The entropy of a cellular automaton cannot be computed or approximated from its local rule \cite{HuKaCu92}, and the set of entropies of one-dimensional cellular automata has a recursion-theoretic characterization as the upper semi-computable ($\Pi_1$) reals \cite{GuZi13}.

Many undecidability results are also known in the reversible setting \cite{Ka08}, for example topological mixing and almost equicontinuity are recursively inseparable properties \cite{Lu10a}. Much is still open: The decidability of \emph{expansivity}, that is whether the orbit of some clopen partition under the RCA generates the topology, is not known for RCA. A classification of entropies is not known, though a class of (computable) entropies of algebraic origin is realized in \cite{Li87}. It is also unknown whether the entropy of an RCA can be approximated based on the local rule, though we mention that the somewhat related Lyapunov exponents were recently shown to be uncomputable from the description of the local rule \cite{Ko19}. Proofs of many undecidability results for RCA are based on the simulation of reversible Turing machines (as dynamical systems \cite{Mo91,Ku97}).

In this paper, we show that reversible cellular automata give rise to new f.g.\ groups with undecidable conjugacy problem. A particularly simple example (obtained from combining our results with \cite{Sa18a}) is described in Figure~\ref{fig:Example} (see Corollary~\ref{cor:Partitioned} for details). One could quite easily imagine bumping into such groups without any a priori interest in their computability properties -- this is what happened. The groups are defined by simple actions on Cantor space (by reversible cellular automata), and the undecidability of their conjugacy is somewhat intrinsic, in that enlargening the group by adding more homeomorphisms cannot turn it decidable. This \emph{eventual local undecidability} of the conjugacy problem may be a new notion, and we exhibit it in several (infinitely-generated) groups. A less technical variant of the same idea proves undecidability of conjugacy for Thompson's $2V$; here we do not show that adding more homeomorphisms cannot turn the problem decidable, but on the other hand $2V$ itself is a very nice and round group.

The concept of f.g.-universality, on top of which we build here, arose in \cite{Sa18a} from studying the groups of RCA that can be built from what we considered the simplest imagineable building blocks among elements of $\Aut(A^\Z)$, namely partial shifts and symbol permutations, see Figure~\ref{fig:Example}. These automorphisms are classical in the theory of cellular automata: a composition of a partial shift with a symbol permutation is often called a \emph{partitioned CA}. The cellular automaton that performs (the school algorithm for) multiplication by $3$ in base $6$ is a prime example of this type, and has been studied both due to its special properties as a CA \cite{Ka12a} and its connection to Mahler's problem \cite{KaKo17}. Composing the subfigures in Figure~\ref{fig:Example} shows the computation $304121_6 \times 3 = 1320403_6$ using this cellular automaton, when $\{0,1\} \times \{0,1,2\}$ is identified with $6$-ary digits in an obvious way.

\begin{figure}
\begin{subfigure}[b]{0.45\textwidth}
\centering
\begin{tikzpicture}[scale=0.55]
\draw (0,0) grid (8,2);
\draw[dashed] (0,0) -- (-1,0);
\draw[dashed] (0,1) -- (-1,1);
\draw[dashed] (0,2) -- (-1,2);
\draw[dashed] (8,0) -- (9,0);
\draw[dashed] (8,1) -- (9,1);
\draw[dashed] (8,2) -- (9,2);
\node () at (0.5,4.5) {$0$};
\node () at (1.5,4.5) {$1$};
\node () at (2.5,4.5) {$0$};
\node () at (3.5,4.5) {$1$};
\node () at (4.5,4.5) {$0$};
\node () at (5.5,4.5) {$0$};
\node () at (6.5,4.5) {$0$};
\node () at (7.5,4.5) {$0$};

\node () at (0.5,3.5) {$0$};
\node () at (1.5,3.5) {$0$};
\node () at (2.5,3.5) {$0$};
\node () at (3.5,3.5) {$1$};
\node () at (4.5,3.5) {$1$};
\node () at (5.5,3.5) {$2$};
\node () at (6.5,3.5) {$1$};
\node () at (7.5,3.5) {$0$};

\node () at (4,2.5) {$\Downarrow$};

\draw (0,3) grid (8,5);
\draw[dashed] (0,3) -- (-1,3);
\draw[dashed] (0,4) -- (-1,4);
\draw[dashed] (0,5) -- (-1,5);
\draw[dashed] (8,3) -- (9,3);
\draw[dashed] (8,4) -- (9,4);
\draw[dashed] (8,5) -- (9,5);
\node () at (0.5,0.5) {$0$};
\node () at (1.5,0.5) {$1$};
\node () at (2.5,0.5) {$0$};
\node () at (3.5,0.5) {$2$};
\node () at (4.5,0.5) {$0$};
\node () at (5.5,0.5) {$1$};
\node () at (6.5,0.5) {$0$};
\node () at (7.5,0.5) {$0$};

\node () at (0.5,1.5) {$0$};
\node () at (1.5,1.5) {$1$};
\node () at (2.5,1.5) {$0$};
\node () at (3.5,1.5) {$0$};
\node () at (4.5,1.5) {$1$};
\node () at (5.5,1.5) {$0$};
\node () at (6.5,1.5) {$1$};
\node () at (7.5,1.5) {$0$};
\end{tikzpicture}
\caption{\emph{Symbol permutations} are permutations $\pi \in \Sym(A)$ applied cellwise; $\pi = ((0,1) \; (1,0) \; (1,1) \; (0,2))$ is shown.}
\end{subfigure}
\hfill
\begin{subfigure}[b]{0.45\textwidth}
\centering
\begin{tikzpicture}[scale=0.55]
\draw (0,0) grid (8,2);
\draw[dashed] (0,0) -- (-1,0);
\draw[dashed] (0,1) -- (-1,1);
\draw[dashed] (0,2) -- (-1,2);
\draw[dashed] (8,0) -- (9,0);
\draw[dashed] (8,1) -- (9,1);
\draw[dashed] (8,2) -- (9,2);
\node () at (0.5,3.5) {$0$};
\node () at (1.5,3.5) {$1$};
\node () at (2.5,3.5) {$0$};
\node () at (3.5,3.5) {$2$};
\node () at (4.5,3.5) {$0$};
\node () at (5.5,3.5) {$1$};
\node () at (6.5,3.5) {$0$};
\node () at (7.5,3.5) {$0$};

\node () at (0.5,4.5) {$0$};
\node () at (1.5,4.5) {$1$};
\node () at (2.5,4.5) {$0$};
\node () at (3.5,4.5) {$0$};
\node () at (4.5,4.5) {$1$};
\node () at (5.5,4.5) {$0$};
\node () at (6.5,4.5) {$1$};
\node () at (7.5,4.5) {$0$};

\node () at (4,2.5) {$\Downarrow$};

\draw (0,3) grid (8,5);
\draw[dashed] (0,3) -- (-1,3);
\draw[dashed] (0,4) -- (-1,4);
\draw[dashed] (0,5) -- (-1,5);
\draw[dashed] (8,3) -- (9,3);
\draw[dashed] (8,4) -- (9,4);
\draw[dashed] (8,5) -- (9,5);

\node[inner sep=1] (z) at (-0.5,0.5) {\phantom{$0$}};
\node[inner sep=1] (a) at (0.5,0.5) {$1$};
\node[inner sep=1] (b) at (1.5,0.5) {$0$};
\node[inner sep=1] (c) at (2.5,0.5) {$2$};
\node[inner sep=1] (d) at (3.5,0.5) {$0$};
\node[inner sep=1] (e) at (4.5,0.5) {$1$};
\node[inner sep=1] (f) at (5.5,0.5) {$0$};
\node[inner sep=1] (g) at (6.5,0.5) {$0$};
\node[inner sep=1] (h) at (7.5,0.5) {\color{gray} $0$};
\node[inner sep=1] (i) at (8.5,0.5) {\phantom{$0$}};

\node () at (0.5,1.5) {$0$};
\node () at (1.5,1.5) {$1$};
\node () at (2.5,1.5) {$0$};
\node () at (3.5,1.5) {$0$};
\node () at (4.5,1.5) {$1$};
\node () at (5.5,1.5) {$0$};
\node () at (6.5,1.5) {$1$};
\node () at (7.5,1.5) {$0$};

\draw[->,color=gray] (a) to[out = -135, in=-45, looseness = 1.5] (z);
\draw[->] (b) to[out = -135, in=-45, looseness = 1.5] (a);
\draw[->] (c) to[out = -135, in=-45, looseness = 1.5] (b);
\draw[->] (d) to[out = -135, in=-45, looseness = 1.5] (c);
\draw[->] (e) to[out = -135, in=-45, looseness = 1.5] (d);
\draw[->] (f) to[out = -135, in=-45, looseness = 1.5] (e);
\draw[->] (g) to[out = -135, in=-45, looseness = 1.5] (f);
\draw[->] (h) to[out = -135, in=-45, looseness = 1.5] (g);
\draw[->,color=gray] (i) to[out = -135, in=-45, looseness = 1.5] (h);
\end{tikzpicture}
\caption{The group $\Z$ acts by \emph{partial shifts}, which shift the ternary track.
\; \; \em \; \; \em \; \; \em \; \; \em \; \; \em \; \; \em \; \; \em \;}
\end{subfigure}
\hfill
\caption{Let $A = \{0,1\} \times \{0,1,2\}$. The conjugacy problem is undecidable in every finitely-generated group of self-homeomorphisms of $A^\Z$, which contains the action of the free product $\Sym(A) * \Z$ generated by (a) symbol permutations and (b) partial shifts.} 
\label{fig:Example}
\end{figure}
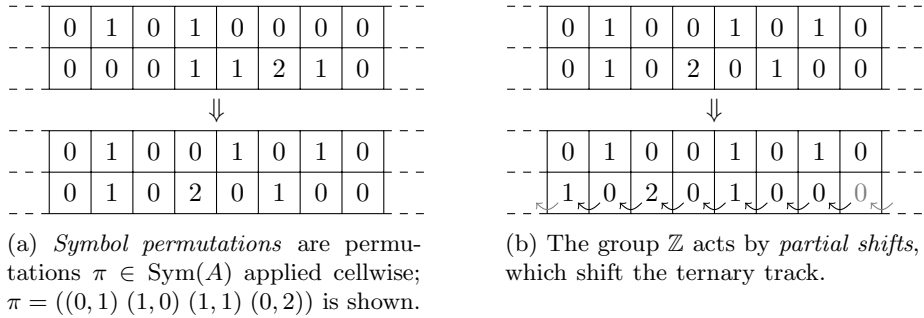

From the point of view of symbolic dynamics, in particular automorphism groups, the main technical contribution of the present paper concerns the finitely-generated subgroup $\CPAut_k[A;B]$ of the automorphism group of $(A \times B)^\Z$ under the action of $\sigma^k$. Our main construction result implies in particular that all marker automorphisms of the system $(B^\Z, \sigma)$ (in the sense of \cite{BoLiRu88}, applied without looking at the $A^\Z$-component) are contained in $\CPAut_k[A;B]$. This is direct containment rather than simulation, so in the terminology of \cite{Sa18a} we obtain that the group generated by marker automorphisms is subfinitely-generated.

Two open problems from the literature are solved: The conjugacy problems for RCA both for the $\Z$-action and the joint $\Z^2$-action with the shift were asked in \cite{JaKa20}. The undecidability of conjugacy in the Brin-Thompson 2V was asked in \cite{BeBl14}.

\subsection{Main results}

By a \emph{topological dynamical system} or \emph{$G$-system} we mean a countable group $G$ acting on a topological space $X$ by homeomorphisms. In the case of topological dynamical systems, a \emph{conjugacy} between $(G, X)$ and $(G, Y)$ usually means dynamical isomorphism, i.e.\ a homeomorphism between $X$ and $Y$ that intertwines the actions. A \emph{full shift} is  the $G$-system $A^G$ for a finite alphabet $A$ and $G$-action $gx_h = x_{g^{-1}h}$. A \emph{subshift} is a closed $G$-invariant subsystem. The \emph{automorphism group} of a subshift consists of its self-conjugacies. We use the term \emph{reversible cellular automaton} (\emph{RCA}) as synonym for an element of an automorphism group of a subshift, especially in the case of a full shift.

An RCA and the shift map taken together give a $\Z^2$-action, and two such $\Z^2$-actions obtained from RCA are dynamically isomorphic if and only if they are conjugate as elements of the automorphism group of the full shift. We can also directly ask whether the $\Z$-actions are dynamically isomorphic.


The following is our main theorem.

\begin{theorem}
\label{thm:Main}
Let $H$ be a group and let $X \subset A^H$ be any subshift that explicitly simulates a nontrivial one-dimensional full shift. Then there exists a finitely-generated subgroup $G \leq \Aut(X)$ with word problem in co-NP, and computable functions $\alpha, \beta : \mathbb{N} \to G$, such that for all $n \in \N$, the following are equivalent:
\begin{enumerate}
\item the $n$th program halts,
\item $\alpha(n) = \beta(n)^g$ for some involution $g \in G$,
\item $\alpha(n) = \beta(n)^g$ for some $g \in \Aut(X)$,
\item the $\Z^2$-systems $(X, \sigma, \alpha(n))$ and $(X, \sigma, \beta(n))$ are isomorphic,
\item the $\Z$-systems $(X, \alpha(n))$ and $(X, \beta(n))$ are isomorphic,
\item $(X, \beta(n))$ embeds into $(X, \alpha(n))$,
\item $\alpha(n)$ and (or) $\beta(n)$ are (is) periodic.
\end{enumerate}
\end{theorem}

``The $n$th program halts'' means $n \in L$, where $L \subset \N$ is (some natural encoding of) the halting problem. The definition of ``explicitly simulates'' is somewhat technical, see Section~\ref{sec:ExplicitSimulation}. Every one-dimensional full shift with explicitly simulates itself, and thus the theorem applies when $X = A^\Z$.

Our result is somewhat orthogonal to \cite{JaKa20}, which proves the undecidability of conjugacy for not necessarily reversible CA. Their full statement includes an entropy separation and that one system is not even a factor of the other when the two CA are not conjugate. It also applies to one-sided cellular automata. On the other hand, the construction of \cite{JaKa20} does not take place inside a f.g.\ subgroup of $\Aut(A^\Z)$ (or submonoid of $\End(A^\Z)$). The fact ours does has corollaries beyond cellular automata and automorphism groups of subshifts, as explained in Section~\ref{sec:Corollaries}.

The proof of the main theorem is based on the undecidability of periodicity of reversible Turing machines (or one-head machines, in the terminology of the present paper), proved in \cite{KaOl08}, so it is not surprising that a slight variant of the construction applies to $2V$. This is proved as Theorem~\ref{thm:2V}, and solves a question of Belk and Bleak \cite{BeBl14}. The elements between which the conjugacy is asked, as well as the conjugating element, are actually elements of the group of reversible Turing machines $\mathrm{RTM}(2,1)$ from \cite{BaKaSa16}, so we get the result also for this group. Intuitively, this is a restriction of $2V$ where one is not allowed to shrink or expand the tape.

\begin{theorem}
The Brin-Thompson group $2V$ has undecidable conjugacy problem. The group of reversible Turing machines $\mathrm{RTM}(s,q)$ has undecidable conjugacy problem unless it is virtually abelian.
\end{theorem}

The group of reversible Turing machines is virtually abelian if and only if the $s = 1$, meaning the tape alphabet is trivial \cite{BaKaSa16}.


We also prove some results about the word problem in groups of RCA which does not seem to appear in the literature. For finitely-generated groups of reversible cellular automata (RCA) we show that this problem is in co-NP, and that it is co-NP-complete for some locally-finite-by-cyclic groups of RCA, see Theorem~\ref{thm:WP}. The proof also adapts to topological full groups of full shifts, see Theorem~\ref{thm:topofullgroup}.

\subsection{Corollaries and alternate wordings}
\label{sec:Corollaries}

From the main theorem, we obtain several corollaries (with mostly trivial proofs), and we discuss these special cases. The third and fourth item of the theorem actually state the same thing in a different language, and give the following corollary:

\begin{corollary}
Let $X$ be a subshift which explicitly simulates a one-dimensional full shift. Then the conjugacy problem in the group $\Aut(X)$ is $\Sigma^0_1$-hard. If additionally $\Aut(X)$ is recursively presented with decidable word problem, then the conjugacy problem is $\Sigma^0_1$-complete.
\end{corollary}

The $\Sigma^0_1$-hardness statement is true when automorphisms are given by local rules, or as words over any sufficiently large (w.r.t.\ set inclusion) set of elements of the group.

The fifth item of the theorem gives that conjugacy of RCA as dynamical systems is undecidable:

\begin{corollary}
Let $X$ be a subshift which explicitly simulates a one-dimensional full shift. Given $f, g \in \Aut(X)$, it is $\Sigma^0_1$-hard to determine whether the $\Z$-systems $(f, X)$ and $(g, X)$ are dynamically isomorphic.
\end{corollary}

The precise complexity of this problem stays open, even when $X$ is a full shift. In this case, the problem is easily seen to be $\Sigma^1_1$. Since many natural sets and ordinals associated to finitely-presented symbolic dynamical systems are high in the arithmetical hierarchy or even properly analytic \cite{We19,Sa13}, one might conjecture that dynamical isomorphism of $\Z$-actions by RCA on full shifts $A^\Z$ is $\Sigma^1_1$-complete.

Next, by the second item of the theorem, the f.g.\ group $G$ itself has undecidable conjugacy; combining this with fifth item shows gives us a nice robust statement.

\begin{corollary}
\label{cor:Robust}
Let $X$ be any subshift that explicitly simulates a full shift. Then there is an f.g.\ group $G \leq \Aut(X)$ such that for any f.g.\ $G'$ group satisfying $G \leq G' \leq \Homeo(X)$ the conjugacy problem is undecidable.
\end{corollary}

As the group $G$, one can take any explicitly f.g.-universal group for $\Aut(A^\Z)$. All the f.g.-universal groups constructed in \cite{Sa18a} are explicitly so, and all embeddings of groups $\Aut(A^\Z)$ into automorphism groups of subshifts that we are aware of (e.g.\ \cite{Ho10}) are explicit. We now give two explicit examples of explicitly f.g.-universal groups.

For finite alphabets $A, B$ define $\PAut'[A; B] \leq \Aut((A \times B)^\Z)$ as the smallest group containing the partial shift $\sigma_2$ on the first track, i.e.\ $\sigma_2(x, y)_i = (x_i, y_{i+1})$, and for each $\pi \in \Sym(A \times B)$ the automorphism $f_{\pi}(x)_i = \pi(x_i)$ (see Section~\ref{sec:Definitions} for details, and Figure~\ref{fig:Example} for an example.).

\begin{corollary}
\label{cor:Partitioned}
Let $A, B$ be nontrivial finite alphabets, with $\max(|A|, |B|) \geq 3$. Then $\PAut'[A;B]$ (and any f.g.\ supergroup in $\Homeo((A \times B)^\Z)$) has undecidable conjugacy problem.
\end{corollary}

\begin{proof}
This follows from the (proof of) Corollary~\ref{cor:Robust} and the proof in \cite{Sa18a} that $\PAut[A; B] = \langle \PAut'[A; B], \sigma_1 \rangle$ is f.g.-universal: As discussed in Section~\ref{sec:ExplicitSimulation}, all known such proofs actually give explicit embeddings, with respect to some explicit simulation of a full shift, and $G$ can be any such group in the proof of Corollary~\ref{cor:Robust}. The shift $\sigma_1$ is easily seen not to be needed in the construction, as the construction is based on constructing a composition of conjugates $f_{\pi}^{\sigma_i}$, and an elementary calculation shows that $f_{\pi}^{\sigma_1} = f_{\pi}^{\sigma_2^{-1}}$.
\end{proof}

For another example, recall that reversible cellular automata have a radius, which determines how far their local rule looks for when deciding the new symbol at each position, see Section~\ref{sec:Definitions} for details. The proof of the following corollary is the same as that of the previous one.

\begin{corollary}
\ref{cor:Partitioned}
Let $A$ be any large enough finite alphabet. Let $F$ be the set of RCA $f \in \Aut(A^\Z)$ such that $f$ and $f^{-1}$ both have radius one. Then $\langle F \rangle \leq \Aut(A^\Z)$ (and any f.g.\ supergroup in $\Homeo(A^\Z)$) has undecidable conjugacy problem.
\end{corollary}

At least $|A| = 6$ and any $|A| \geq 32$ are large enough by \cite{Sa18a}. As mentioned in the introduction, the fact that that there exists a f.g.\ subgroup of $\Aut(A^\Z)$ with $\Sigma^0_1$-hard conjugacy problem is not new: $F_2 \times F_2$ embeds into the automorphism group of a full shift \cite{KiRo90} and has such subgroups \cite{Mi68a}. However, by construction this cannot be used in Theorem~\ref{thm:Main}, since $F_2 \times F_2$ itself has a decidable conjugacy problem, and Corollary~\ref{cor:Robust} does not seem obtainable this way.


We also make explicit the subshift-free corollary:

\begin{corollary}
Let $C$ be a Cantor set. There exist $f, g \in \Homeo(C)$ such that every f.g.\ subgroup $F \leq \Homeo(C)$ containing $f$ and $g$ has undecidable conjugacy problem.
\end{corollary}

\begin{proof}
This follows from the main theorem by taking any $2$-generated f.g.-universal $G$, since $A^\Z$ is a Cantor set. One is constructed in \cite{Sa20a}.
\end{proof}

It is easy to implement the natural action of our f.g.-universal groups inside the asynchronous rational group $Q$ defined in \cite{GrNe00} over the alphabet $A^2$, by unfolding $(A^2)^\omega$ into $A^\Z$ and acting through this interpretation. A stronger result is \cite[Theorem~3]{GrNe00}. By a standard encoding argument, we can do the same for any nontrivial alphabet. Since asynchronous rational transductions (with finite lag) are homeomorphisms of a Cantor space, we obtain an analog of the above corollary for this group.


If $\mathcal{C}, \mathcal{D}$ are properties of groups, say $G$ is \emph{($\mathcal{D}$-)eventually locally $\mathcal{C}$} 
if there exists a finitely-generated subgroup $H \leq G$ (with property $\mathcal{D}$), such that every finitely-generated $H' \leq G$ containing $H$ has property $\mathcal{C}$. Of course, if $G$ is finitely-generated, it is eventually locally $\mathcal{C}$ if and only if it is $\mathcal{C}$. We say \emph{$G$ has ($\mathcal{D}$-)eventually locally undecidable conjugacy problem} if it is ($\mathcal{D}$-)eventually locally $\mathcal{C}$ where $\mathcal{C}$ is the class of groups with undecidable conjugacy problem.

\begin{theorem}
Let $\mathcal{D}$ be the class of $2$-generated groups. The following groups have $\mathcal{D}$-eventually locally undecidable conjugacy problem:
\begin{itemize}
\item the automorphism group of any full shift,\footnote{or any subshift that explicitly simulates a full shift}
\item the asynchronous rational group over any alphabet, and
\item the homeomorphism group of Cantor space.
\end{itemize}
\end{theorem}

It seems likely that there is an easier proof of this result in the case of the Cantor space, but we have not found this statement in the literature. 

\section{Definitions and conventions}
\label{sec:Definitions}

Our conjugation notation is $h^g = g^{-1}hg$, and commutators are $[g, h] = ghg^{-1}h^{-1}$. We include $0 \in \N$. Everything is $0$-indexed by default. For $a, b \in \Z$ we define $\llb a, b \rrb = [a,b] \cap \Z$ (where $[a,b]$ denotes the real interval). An (always finite) alphabet $A$ is a discrete set; it is \emph{nontrivial} if $|A| \geq 2$. We assume some knowledge of group theory and recursion theory, and we refer to \cite{ArBa09} for definitions of complexity theoretic classes (P, co-NP, PSPACE). We write $\langle \cdots \rangle$ to denote the smallest group containing $\cdots$, which, abusing notation, can include both elements and sets (whose elements we include).

A full shift $A^G$ (always with Cantor topology and finite alphabet $A$) is \emph{nontrivial} if the alphabet $A$ has cardinality at least $2$. A \emph{factor} of a $G$-system is a shift-commuting continuous image. Subshifts are, as defined in the introduction, closed $G$-invariant subsets of $A^G$. Sometimes it is helpful to think of subshifts more abstractly as expansive actions on zero-dimensional compact metrizable spaces (\cite{LiMa95}).

We write dynamical $\Z$-systems also as $(X, f)$ where $f : X \to X$ is a homeomorphism whose iterations give the action. For one-dimensional subshifts $A^\Z$ we use the $\Z$-action given by the homeomorphism $\sigma(x)_i = x_{i+1}$ (this is the natural action of $-i$ in the sense of the general definition). For $\Z^2$-actions, we also write $(X, f, g)$ where $f, g : X \to X$ are the generators.

For a finite alphabet $A$ write $A^*$ for the set of all \emph{words} over $A$, i.e.\ the free monoid generated by $A$ (including the empty word $\epsilon$). We write word concatenation as literal concatenation $(u,v) \mapsto uv$. A \emph{language} is a subset of $A^*$. \emph{Regular expressions} over alphabet $A$ are a convenient language for expressing languages over $A$. Each $a \in A$ is itself a regular expression denoting the language $\{a\}$, and for languages $K,L$ write $KL = \{uv \;|\; u \in K, v \in V\}$, $L^* = \{\epsilon\} \cup \{w_1w_2 \cdots w_k \;|\; \forall i: w_i \in L\}$, and $K+L = K \cup L$. The empty language is denoted by $\emptyset$.

We extract words from $x \in A^\Z$ by $x_{\llb a, b \rrb} = w$ where $w_i = x_{a+i}$ for applicable $i$. \emph{The language} of a one-dimensional subshift $X \subset A^\Z$ is the set of words $w \in A^*$ such that for some $x \in X$ we have $x_{\llb i,i+|w|-1 \rrb} = w$ for some $i$. The language of a point $x \in X$ is the language of the subshift $\overline{O(x)}$. We say a word $w$ \emph{occurs} in $x \in A^\Z$ or $x$ \emph{contains} $w$ if $w$ is in the language of $x$; a subshift $X$ contains $w$ if a point of it does.

When dealing with words over product alphabets $w \in (A \times B)^n$ or configurations over such $x \in (A \times B)^\Z$, we always (often implicitly) use the natural bijection $(A \times B)^N \cong A^N \times B^N$ that exists for any set $N$. This should not cause any confusion. We refer to the first element of $A^N \times B^N$ as the (contents of the) \emph{top track}, and the second one is the \emph{bottom track}. It is useful to think of the tape as arranged geometrically on $\Z \times \{0,1\}$ with $\Z$ being the horizontal axis and $\{0,1\}$ laid vertically, $1$ on top, with the $1$-components containing the elements of $A$ and the $0$-components those of $B$.

As defined in the introduction, the \emph{automorphisms} of a $G$-system are its $G$-commuting self-homeomorphisms. The automorphisms $f : X \to X$ of a subshift $X \subset A^\Z$ have a \emph{radius} $r \in \N$ and a \emph{local rule} $F : A^{2i+1} \to A$ such that $f(x)_i = F(x_{\llb i-r, i+r \rrb})$ for all $x \in X$ and $i \in \Z$. The \emph{biradius} is $r$ if both $f$ and $f^{-1}$ have radius $r$. When $X = A^\Z$ we also call automorphisms \emph{reversible cellular automata} or \emph{RCA}.

Two particular types of RCA are the \emph{partial shifts} $\sigma_1(x,y) = (\sigma(x), y)$ and $\sigma_2(x,y) = (x, \sigma(y))$ and \emph{symbol permutations} $f(x)_i = \pi(x_i)$ where $\pi \in \Sym(A)$. The group $\CPAut_k[A;B]$ is the subgroup of $\Aut(((A \times B)^\Z, \sigma^k))$ obtained from the partial shifts $\sigma_i$ with respect to $\sigma$, and symbol permutations with respect to $\sigma^k$, i.e.\ the group generated by the set
\[ T = \{\sigma_1, \sigma_2\} \cup \{f_{\pi} \;|\; \pi \in \Sym((A \times B)^k)\}, \]
where $f_{\pi}(x)_{\llb ki,ki+n-1 \rrb} = \pi(x_{\llb ki,ki+n-1 \rrb})$.

A set of words $W \subset A^*$ for a finite alphabet $A$ is \emph{mutually unbordered} if $uw = w'v$ for $w, w' \in W$ implies that either $|u| \geq |w'|$ or $|u| = |v| = 0$ (and thus $w = w'$). A \emph{marker automorphism} is determined by a mutually unbordered set of words $W \subset A^n$ and some $\pi \in \Sym(W)$, and is the unique automorphism $g \in \Aut(A^\Z)$ such that $g(x)_i = \pi(w)_j$ if $x_{\llb i-j,i-j+n-1 \rrb} = w \in W$ for some $j < n$, and $g(x)_i = x_i$ when such $w$ does not exist. It is easy to see that arbitrarily large mutually unbordered sets $W \subset A^n$ exist, when $A$ is nontrivial. A word $w$ is \emph{unbordered} if $\{w\}$ is mutually unbordered. Say a word $w \in A^*$ has \emph{period} $n$ if $w_i = w_{i+n}$ for all applicable $i$ (possibly none).  The length of a word is always a period of it, and a word is unbordered if and only if its minimal period is its length.

For $u \in A^*$, a \emph{cylinder} is $[u]_i = \{x \in A^\Z \;|\; x_{i+j} = u_j \mbox{ for all applicable $j$}\}$. Write also $[u]_0 = [u]$ and for $W \subset A^n$ write $[W]_i = \bigcup_{w \in W} [w]_i$. The topology of $A^\Z$ has the basic cylinders $[a]_i, a \in A, i \in \Z$, as subbasis. A clopen set $D \subset X$ is \emph{$n$-aperiodic} if $D \cap \sigma^i(D) \neq \emptyset \implies i = 0 \vee |i| \geq n$. If $W \subset A^n$ is mutually unbordered, $[W]$ is $n$-aperiodic. Recall that the clopen sets of the Cantor space $A^\Z$ are just the finite unions of cylinders. We say a clopen set $C \subset A^\Z$ \emph{occurs} at $i$ in $x$ if $\sigma^i(x) \in C$; the intuition is that one of the words defining $C$ as a union of cylinders occurs near $i$ (with the correct offset).

To avoid confusion between Turing machines as dynamical systems, Turing machines as a model of computation, and Turing machines as elements of groups of Turing machines in the sense of \cite{BaKaSa16}, we will say \emph{$n$th program halts} to refer to instances of the halting problem, use the term \emph{one-head machine} to refer to Turing machines as dynamical systems, and always mention \cite{BaKaSa16} when using this word in the third sense.

We assume some familiarity with Turing machines/one-head machines as dynamical systems \cite{Ku97}, in particular the moving-head and moving-tape models. All machines studied in this paper are ``complete'' and ``reversible''. As a quick recap, in the moving-head model, a \emph{(reversible) one-head machine} with state set $Q$ and tape alphabet $S$ is a particular type of self-homeomorphism of the Cantor space $(Q \times \Z \times S^\Z) \cup S^\Z$. In the topology, $(q,i,x)$ is close to $x$ whenever $|i|$ is large and $S^\Z$ is the compactification of $Q \times \Z \times S^\Z$ where $|i| \rightarrow \infty$; one should think of a \emph{head} $q \in Q$ as being in coordinate $i \in \Z$, on top of the tape symbol $x_i$. The tape is never shifted, and rather the head moves around on the tape and may modify the symbol under it; this action can depend on the current value of $Q$ and the symbol $x_i$. In the moving-tape model, the same machine acts on $Q \times S^\Z$, and movement of the head translates into shifting the tape.

As the specific model of one-head machine one can take the notion of \emph{complete RTM} from \cite{KaOl08}, so a one-head machine is a triple $(Q,S,T)$ where $T \subset (Q \times \{-1,1\} \times Q) \cup (Q \times S \times Q \times S)$, satisfying some axioms. We omit the precise definitions, because in the only proof where we need to discuss the specifics of the machine's behavior (item three in Lemma~\ref{lem:Mn}), we only make minor modifications to the construction of \cite{KaOl08} and only need some high-level properties of its (dynamical) behavior. (Alternatively, one could take the more general and dynamically natural definition from \cite{Mo91,BaKaSa16}; this is what the title of this paper refers to.)

\subsection{Explicit simulation of full shifts}
\label{sec:ExplicitSimulation}

Essentially all known embeddings between automorphism groups of subshifts are performed by ``explicit simulation''. It is best to understand what this means on an I-know-it-when-I-see-it basis, through reading any of the embedding proofs (e.g.\ \cite{KiRo90,Sa18d}), all of which involve explicit markers representing the symbols of another subshift (full shift). For completeness, we give a formal definition. We restrict to simulation of one-dimensional subshifts by subshifts over any group. One could easily extend this to simulation of subshifts on groups, and to zero-dimensional (not necessarily expansive) systems.

\begin{definition}
\label{def:ExplicitSimulation}
Let $G$ be a group, and let $X \subset C^G$ and $Y \subset B^\Z$ be subshifts. We say $X$ \emph{explicitly simulates} $Y$ if $X$ is conjugate to a subshift $X' \subset A^G$ such that the statements in the following four paragraphs hold:

To each $a \in A$ is associated a number $m(a) \in \N$, a word $w(a) \in B^{m(a)}$, an element $v(a) \in (G \times \N)^{m(a)}$, and some additional piece of information $\i(a) \in I$ (for some finite set $I$). Furthermore, $a$ is uniquely determined by the tuple $(m(a), w(a), v(a), \i(a))$. To simplify the following discussion, we identify $A$ directly by its image and assume $A \subset \N \times B^* \times (G \times \N)^* \times I$ (of course, finite).

To each $x \in X'$ we can associate a directed graph by taking nodes $V(x) = \{ (g, j) \;|\; g \in G, j < m(x_g) \}$, and for each $(g, j) \in V(x)$ with $v(x_g)_j = (h, j')$, adding an edge from $(g, j)$ to $(gh, j')$ (so we must have $(gh, j') \in V(x)$). If $(g,j) \in V(x)$, then the \emph{weakly connected component} of $(g,j)$ in this graph, i.e.\ the connected component in the undirected graph obtained by forgetting the edge directions, is isomorphic to either $\Z$ or a finite cycle (this is equivalent to every node having in-degree one). Further, if (read cyclically) this path is $\cdots, (g_{-1},j_{-1}), (g_0, j_0), (g_1, j_1), \cdots$, then defining $y \in B^\Z$ by $y_k = w(x_{g_k})_{j_k}$ we have $y \in Y$. We call this the \emph{simulated configuration at $(g, j)$ (on $x$)}, and if the cycle is finite we call its cardinality its \emph{structural period}, otherwise the structural period is $0$.

We require that the simulated configuration can be changed arbitrarily: suppose $x \in X'$, $(g,j) \in V(x)$, and the simulated configuration at $(g,j)$ is $y \in Y$, and the $(g_k, j_k)$ are as in the previous paragraph. If the structural period is zero, then for any $y' \in Y$, we have $x' \in X'$ where $x'$ is defined by replacing $w(x_{g_k})_{j_k} = y_k$ by $y'_k$ for all $k \in \Z$. If the structural period is $p \geq 1$, then we require the same for any $p$-periodic point $y' \in Y$.

Finally, we require that every configuration of $Y$ is simulated: for any $y \in Y$, there must exist $x \in X'$ and $(g,j) \in V(x)$ such that the simulated configuration at $(g,j)$ is $y$.
\end{definition}

In the presence of others, for the last requirement it suffices that some $y \in Y$ is simulated without a structural period.

In \cite{Sa18d}, it is proved that if $X$ is an uncountable sofic shift, then $\Aut(A^\Z) \leq \Aut(X)$ for any finite alphabet $A$. In the proof, two symbols of a full shift are simulated between two unbordered words at a particular distance from each other, and when the encoding breaks, we wrap the simulated configuration into a conveyor belt. From the proof, one easily obtains the following:

\begin{lemma}
\label{lem:UncountableSimulatesFull}
Every uncountable sofic shift explicitly simulates every full shift.
\end{lemma}

\begin{proof}
Let the sofic shift $X$ be over alphabet $C$, and let $D$ be the alphabet of the full shift we want to simulate. We pick the alphabet of $X'$ to be $A \supset C$, and set $m(a) = \i(a) = 0$ for all $a \in C$. The conjugacy from $X$ to $X'$ typically preserves the symbol $a \in A$.

Besides $C$, in $A$ we have symbols $\#$ and $D^2$. Following the proof from \cite{Sa18d} (or \cite{KiRo90}), between unbordered words such that the word in between codes some $(d,d') \in D^2$, we replace it by a word of the form $(d,d')\#^*$. Clearly this is a conjugacy, as we can return $(d,d')\#^*$ back to the coding word by a local rule. We set $m(\#) = \i(\#) = 0$, and $m((d,d')) = 2, w((d,d')) = (d,d')$.

The vector $v((d,d'))$ is determined by detecting whether the sequence of coding words continues: if it continues to the right, set $v((d,d'))_0 = (t, 0)$ where $t$ is the distance between coding segments (to continue to the neighbor on the right, on the $0$-track), and if it does not continue set $v((d,d'))_0 = (0, 1)$ (to connect to the same cell's $1$-track). The value of $v((d,d'))_1$ is determined similarly, but replacing $t$ by $-t$ and exchanging the roles of the tracks.
\end{proof}

The construction in \cite{Ho10} shows that every $\Z^d$-SFT with positive entropy and dense minimal points explicitly simulates every (one-dimensional) full shift, for any $d \geq 1$.

The importance of explicit simulation of another subshift is that it allows explicit simulations of automorphisms: Let $G$ be a group, and let $X \subset C^G$ and $Y \subset B^\Z$ be subshifts. Let $M \leq \End(Y)$ and $N \leq \End(X)$ be monoids. If $X$ explicitly simulates $Y$, we can apply each $h \in M$ on simulated configurations in an obvious way (through the conjugate view $X'$, from the definition of explicit simulation). This construction maps the identity map to the identity map, and thus maps automorphisms to automorphisms.

\begin{definition}
Any embedding $\phi : M \to \End(X)$ with $\phi(M) \subset N$ of the kind described in the previous paragraph (for two submonoids $M, N$ of endomorphism monoids of subshifts) is called an \emph{explicit embedding} of $M$ inside $N$.
\end{definition}

We usually just write that $\phi : M \to N$ is an explicit embedding in this case.

We say a group $G \leq \Aut(A^\Z)$ is \emph{f.g.-universal} if every f.g.\ subgroup $H \leq \Aut(A^\Z)$ embeds into $G$. In \cite{Sa18a} it is shown that such groups exist. The embeddings constructed in the proof (see \cite[Lemma~7]{Sa18a}) are of the explicit kind: We have two tracks, and under each unbordered words of length $24r$ on the top track, we simulate four $6r$-length subwords of a full shift, with similar wrapping behavior as above. As in the proof of Lemma~\ref{lem:UncountableSimulatesFull}, one can conjugate this to an explicit simulation by explicitly marking the graph structure of the conveyor belts. Furthermore, the construction is algorithmic (as noted in \cite{Sa18a}), and we obtain the following from \cite[Lemma~6~and~7]{Sa18a}:

\begin{lemma}
\label{lem:GDerp}
Let $X = A^\Z$ for a nontrivial alphabet $A$. Then there exists a f.g.\ group $G \leq \Aut(X)$ such that for any full shift $B^\Z$, and any f.g.\ group $H \leq \Aut(B^\Z)$, there is a computable explicit embedding $H \leq G$.
\end{lemma}

We call such groups $G$ \emph{explicitly f.g.-universal}. A subtlety in this result is that while $G$ and $X$ stay fixed as $H$ ranges over f.g.\ subgroups of $\Aut(B^\Z)$, the conjugate subshift $X'$ in the definition of explicit simulation depends on the subgroup $H$.

In the following lemma, the first item is true by definition. The second and third take a bit of work, but we omit the proofs.

\begin{lemma}
\label{lem:Transitive}
Let $X,Y,Z$ be $\Z$-subshifts. Let $G \leq \Aut(X), H \leq \Aut(Y), K \leq \Aut(Z)$.
\begin{itemize}
\item If $H$ explicitly embeds into $G$ and $H' \leq H$, then $H'$ explicitly embeds in~$G$.
\item If $X$ explicitly simulates $Y$ and $Y$ explicitly simulates $Z$, then $X$ explicitly simulates $Z$.
\item If $K$ explicitly embeds into $H$ and $H$ explicitly embeds into $G$, then $K$ explicitly embeds into $G$,
\end{itemize}
\end{lemma}

Combining these lemmas, we obtain the following:

\begin{lemma}
\label{lem:G}
Let $X$ be any subshift that explicitly simulates a nontrivial full shift. Then there exists a f.g.\ group $G \leq \Aut(X)$ with co-NP word problem, such that for any full shift $B^\Z$, and any f.g.\ group $H \leq \Aut(B^\Z)$, there is a computable explicit embedding $\phi : H \to G$.
\end{lemma}

\begin{proof}
Since $X$ simulates a nontrivial full shift, it simulates every nontrivial full shift by Lemma~\ref{lem:UncountableSimulatesFull} and the second item of Lemma~\ref{lem:Transitive}. In particular, it simulates $A^\Z$ such that Lemma~\ref{lem:GDerp} applies, and there is an explicit embedding $\Aut(A^\Z) \leq \Aut(X)$. By Lemma~\ref{lem:GDerp}, there exists a f.g.\ group $G' \leq \Aut(A^\Z)$ such that for any full shift $B^\Z$, and any f.g.\ group $H \leq \Aut(B^\Z)$, there is a computable explicit embedding of $H$ into $G'$.

From the first item of Lemma~\ref{lem:Transitive} we then obtain that $G'$ explicitly embeds into $\Aut(X)$, and taking $G$ to be the image of this embedding $G$ has word problem in co-NP (because $G' \cong G$ does, by Theorem~\ref{thm:WP}). From the third item we then obtain that for any full shift $B^\Z$, and any f.g.\ group $H \leq \Aut(B^\Z)$, $H$ explicitly embeds into $G$, as required, and this embedding is computable, since group-theoretically it is the same embedding as that of $H$ into $G'$ (i.e.\ the decomposition into generators is the same).
\end{proof}

\section{Conjugacy of reversible cellular automata}

\subsection{The main construction}

In this section we construct the automorphisms $\alpha(n)$ and $\beta(n)$ used to prove the main theorem (Theorem~\ref{thm:Main} from the introduction). The basic construction is done on a full shift (whose alphabet depends on $n$). First, we recall from \cite{KaOl08} that periodicity of (reversible) one-head machines is undecidable. We make some modifications to the construction, and sketch the proof of a dynamical property of the resulting machines (the third item in the lemma), which is clear from the construction of \cite{KaOl08}. This property is very robust to implementation details (up to possibly taking a power of the machine), and seems like a necessary consequence of any Hooper-like \cite{Ho66} construction.

\begin{lemma}
\label{lem:Mn}
Given $n$, we can effectively construct a one-head machine $M_n = (Q, S, T)$ such that the $n$th program halts iff $M_n$ is periodic in the moving head model. Furthermore, we may assume that
\begin{itemize}
\item the machine $M_n$ only moves on every fourth step: $Q$ has as a direct summand an explicit mod-$4$ \emph{counter} that alternates between an \emph{active} step and three \emph{inactive} steps, and we move the machine and make modifications to the tape only when it is active,
\item $Q$ has as direct summands two explicit \emph{dummy bits},
\item if the $n$th program never halts, then there exist $x \in S^\Z, q \in Q, i \in \Z$ such that started from $(q, i, x)$, in the moving head model,
\begin{itemize}
\item for all $t$ there exist $t_1 < t_2 < t_3 < t_4$ with $t_2 - t_1, t_4 - t_3 \geq t$ such that the coordinate $0$ is visited by the head an odd (finite) number of times in the time interval $\llb t_2, t_3 \rrb$ over the $M_n$-orbit, with the counter in the active state, and
\item during the time intervals $\llb t_1, t_2-1 \rrb$, $\llb t_3+1, t_4 \rrb$, the machine $M_n$ does not visit the origin.
\end{itemize}
\item $|Q| \geq 5$.
\end{itemize}
\end{lemma}

\begin{proof}
By \cite[Theorem~8]{KaOl08}, we can effectively construct a machine $M_n$ that is periodic if and only if the $n$th program halts. The first two properties are self-explanatorily fulfilled, simply replace $Q$ by
\[ Q \times \{\mbox{active},\mbox{inactive}_1, \mbox{inactive}_2, \mbox{inactive}_3\} \times \{0,1\}^2 \]
and make the obvious modifications to the rule: dummy bits are ignored, and the counter is cycled, and we only take a step of the original machine when the counter is in the active state. The third property follows automatically from the construction of \cite{KaOl08}; in fact, the odd number is $1$, and the origin is visited only once over the entire orbit. We now briefly outline why this is the case (assuming the $n$th program never halts).

We assume the reader has a high-level understanding of the construction, but recall the basic idea: the $n$th program whose halting we are interested in is first simulated on a reversible two-counter machine in a standard way, and a (reversible) one-head machine simulates this counter machine on words of the form $@ 1^a x 2^b y$ (simulating counter machine configurations $(s,a,b)$ for some counter machine state $s$ determined by the state of the simulating one-head machine). The way the counter machine is simulated is by bouncing back and forth between the left and right end of $@ 1^a x 2^b y$, using the internal state of the one-head machine to simulate the counter machine, to move the counters, and to check their positivity.

The nontrivial issue to overcome is that arbitrarily long searches are needed when $a$ or $b$ is large, leading to infinite orbits. To combat this, whenever we do not immediately find the delimiter we are looking for ($@$, $x$ or $y$), we lay down $@_{\delta}xy$ (where $\delta$ records the current state) and begin another simulation from an initial state. When we run out of space for this simulation (by hitting another delimiter or some encoding error), we simply reverse the direction of simulation; for this the state should contain a direction bit that determines the direction of simulation. When a backward computation hits $@_{\delta}xy$ in an initial state, we erase this pattern from the tape and continue the computation from state $\delta$.

Now, consider what happens when the one-head machine is on the $@_\delta$ symbol on a subword of the form $@_\delta 1^a x 2^b y$, for large $a, b$ and some state $\delta$, and the machine itself is in state $\gamma$. The one-head machine is simulating a computation of some reversible counter machine on some configuration $(s, a, b)$, where $s$ is determined by the current state. Now, quoting \cite{KaOl08}, ``A simulation of one move of the CM consists of (1) finding delimiters ``$x$'' and ``$y$'' on the right to check if either of the two counters is zero and (2) incrementing or decrementing the counters as determined by the CM.'' This high-level description suffices for us.\footnote{In theory, it could be that the implementation of this high-level description actually always makes the head visit the origin an even number of times between steps (1) and (2) of a simulated counter machine step. This is not the case, and even if it were, assuming any reasonable implementation details, we would obtain our desired statement by simulating instead a high enough power of the machine.} Suppose the current state $\gamma$ is such that we are simulating the CM forward in time, and are between these two steps (1) and (2), meaning the machine has just returned from checking whether the counters are zero, and next proceeds to increment or decrement a subset of them.

Then, both forward and backward in time, the machine steps to the right of $@_\delta$, places some $@_\varepsilon x y$ on the tape (over the $1^a$ area), and begins a new computation. This computation does not step back on the original $@_\delta$-symbol until we run out of space for the computation in the $1^a$-area, which necessarily takes at least $a-O(1)$ steps.

By compactness, on any configuration $\eta = (\gamma,0,x)$ with $x_{\llb0,\infty)} = @_\alpha 1^\omega$, the head is at the origin only on time-step $0$. Thus any such $\eta$ satisfies the third property.

The fourth property is automatic.
\end{proof}



We now describe automorphisms $\alpha'(n)$ and $\beta'(n)$ (from which $\alpha(n)$ and $\beta(n)$ will be built). These are automorphisms of the same full shift $B^\Z$ (depending on $n$) where 
\[ B = \{0,1\}^2 \cup A \; \mbox{  and  } \; A = (S^2 \times \{\leftarrow, \rightarrow\}) \cup (Q \times S) \cup (S \times Q) \] 
where $Q$ and $S$ are (resp.) the state set and tape alphabet (resp.) of the machine $M_n$. The elements of $\{0,1\}^2$ are called \emph{blinkers}, and the two bits are respectively called the \emph{activity bit} (which can be \emph{active} $= 1$ or \emph{inactive} $=0$) and the \emph{blink bit}. 


The automorphism $\alpha'(n)$ works in three steps (meaning it is the composition of three automorphisms). First, it applies the standard conveyor belt simulation of $M_n$ as described in \cite[Lemma~3]{GuSa17} (the blinkers are simply ignored in this simulation, and each blinker cuts the tape in two). We refer to the proof of \cite[Lemma~3]{GuSa17} for details, but the basic idea is that if a maximal (possibly one- or two-sided infinite) word over the alphabet $A$ contains a unique \emph{head}, i.e.\ element of $(Q \times S) \cup (S \times Q)$, and all arrows $\{\leftarrow, \rightarrow\}$ point toward it, then we interpret it as either periodic point with repeating pattern in $QS^{2n-1}$ for some $n$, or as $S^{-\omega}QS^\omega$; to do this, the bottom track is read in reverse, and at the possible boundaries of the word, we join the top and bottom tracks like a conveyor belt. These words where the one-head machine is simulated are called \emph{segments}, and a simple local rule can identify the segments.

In the second step of $\alpha'(n)$, every blinker in an active state flips its blink bit (from $b$ to $1-b$), and inactive blinkers do nothing. In the third step, if a head is on the top track (so $(q, s) \in Q \times S$) in an active state $q$ (meaning the counter state is active), and there is a blinker immediately to the left of it, then the first dummy bit of the state of $M_n$ is flipped (that is, we replace $q \in Q$ by the state which is identical to $q$ except that the dummy bit is flipped).

The automorphism $\beta'(n)$ first applies $\alpha'(n)$, and then an additional automorphism which flips the activity of a blinker under the same condition under which the first dummy bit of a head was flipped in the third step of $\alpha'(n)$: if immediately to the right of a blinker there is a head $(q, s) \in Q \times S$, and $q$ is an active state of $M_n$, then the activity of the blinker is flipped. Note that neither automorphism touches the second dummy bit of the state. 

The above discussion serves as the definition of $\alpha'(n)$ and $\beta'(n)$. We omit the formula for the local rule, but it is in principle straightforward to write down a local rule based on this description.

From $\alpha'(n)$ and $\beta'(n)$ we build two automorphisms $\alpha''(n)$ and $\beta''(n)$ of the full shift $(A \times B)^\Z$, by applying them on the second track and fixing the first track, so e.g.\ $\alpha''(n)(x,y) = (x, \alpha'(n)(y))$. Note that in particular $\alpha''(n)$ and $\beta''(n)$ commute with $\sigma^2$, and we can also see them as elements of $\Aut(((A \times B)^\Z, \sigma^2))$. 


In the proof of the main theorem, we need to construct $\alpha(n)$ and $\beta(n)$ inside a finitely-generated group $G$. The group we will use is that from Lemma~\ref{lem:G}. The automorphisms $\alpha(n)$ and $\beta(n)$ are constructed inside $G$ by explicitly simulating $\alpha''(n)$ and $\beta''(n)$, but with the small subtlety that we also simulate some additional elements: let $T$ be the generating set of $\CPAut_2[A;B]$, then take
\[ H' = \langle \alpha''(n), \beta''(n), T \rangle. \]
Now, construct an explicit embedding $\phi : H' \to H(n) \leq G$ (surjective onto a subgroup $H(n)$), and let $\alpha(n)$ and $\beta(n)$ be the images of $\alpha''(n)$ and $\beta''(n)$ in this embedding.

The names of the objects that need to be kept in mind are summarized in the following definition.

\begin{definition}
\label{def:Remember}
Suppose a subshift $X$ and a subgroup $G \leq \Aut(X)$ with the properties stated in Lemma~\ref{lem:G} are fixed. In this section we have associated to every $n \in \N$ a one-head machine $M_n$, automorphisms $\alpha'(n), \beta'(n) \in \Aut(B^\Z)$ for some alphabet $B = B(n)$, automorphisms $\alpha''(n), \beta''(n) \in \Aut((A \times B)^\Z)$, an f.g.\ subgroup $H(n) \leq G$ and automorphisms $\alpha(n), \beta(n) \in H(n)$, all effectively constructed from $n$.
\end{definition}

\subsection{What to do when the $n$th program halts}
\label{sec:WhatToDoWhenHalts}

In this section, we describe the conjugacy between $\alpha'(n)$ and $\beta'(n)$ when the $n$th program halts. Suppose thus that it does, so that $M_n = (Q, S, T)$ is periodic in the moving-head model. We first make some dynamical observations. (The symbols from Definition~\ref{def:Remember}, in particular $A$ and $B$, retain their meaning.)

Fix linear orders on $Q$ and on $S$. Then a (somewhat arbitrary) preorder is obtained on $Q \times \Z \times S^\Z$ by setting $(q^1,k^1,x^1) < (q^2,k^2,x^2)$ if $q^1 < q^2$, or $q^1 = q^2$ and $y^1 < y^2$ in lexicographic order where
\[ y^i = x^i_{k^i} x^i_{k^i+1} x^i_{k^i-1} x^i_{k^i+2} x^i_{k^i-2} x^i_{k^i+3} \ldots \]
Let $\eta \sim \eta'$ if $\eta \not< \eta'$ and $\eta' \not< \eta$.
If $A \subset Q \times \Z \times S^\Z$ we say $\eta \in A$ is \emph{strictly minimal} in $A$ if $\eta < \eta'$ for all $\eta' \in A \setminus \{\eta\}$, and \emph{minimal} if $\eta < \eta'$ or $\eta \sim \eta'$ for all $\eta' \in A \setminus \{\eta\}$.

Let $C \subset Q \times \Z \times S^\Z$ be the set of configurations $\eta = (q,k,x)$ satisfying the following properties:
\begin{itemize}
\item the head is at the origin ($k = 0$),
\item $\eta$ is strictly minimal in the orbit $\{M_n^i(\eta) \;|\; i \in \Z\}$ with respect to the order defined above.
\end{itemize}
It is easy to prove that $C$ is clopen: Any minimal element in the orbit is automatically strictly minimal, because $\eta = (q,0,x) \sim \eta'$ for $\eta \neq \eta'$ implies $\eta' = (q,k,\sigma^{-k}(x))$ for some $k \neq 0$, and for two configurations in the same orbit this would imply aperiodicity in the moving-head model (and the machine could not be periodic). Clopenness then follows because minimality and ``not being strictly minimal'' over the (bounded) orbit are clearly closed conditions.

Let now $k$ be such that whether $(q, 0, x) \in C$ or not is determined by the pair $(q, x_{\llb -k,k \rrb})$. 
The set $C$ can also be interpreted in the moving-tape model (dropping the $0$s from the triples), and it is then a cross-section for the dynamics: every configuration enters it at a syndetic set of times. When it recurs in the moving-tape model, it must recur also in the moving-head model, since the one-head machine periodic in the moving-head model.

We now conjugate $\alpha'(n)$ to $\beta'(n)$ by a sequence of commuting marker permutations, which are in fact even involutions (thus their composition is also an even involution). We start with an elementary lemma about symmetric groups of products. If a group $G$ acts on a product $A \times B$, we say its action is \emph{well-defined on $A$} if for all $g \in G$ we have $g(a,b) = (c,d) \wedge g(a,b') = (c',d') \implies c = c'$, and its \emph{action on $A$} is then defined by $ga = a' \iff \exists b,b': g(a,b) = (a',b')$.

\begin{lemma}
Let $W_1, W_2, \cdots W_m$ be finite sets with $|W_i| \geq 5$ for all $i$. Suppose that for each $j$ we have an action of $\Alt(W_j)$ on $\prod_i W_i$ which does not modify the coordinates $i < j$, has a well-defined action on $W_j$, which is simply the natural action, and may modify the components $W_i$ for $i > j$ arbitrarily. Then the induced action of the free product of the groups $\Alt(W_j)$ contains the natural action of $\Alt(W_j)$ on the $j$th coordinate, for all $j$.
\end{lemma}

\begin{proof}
Let $G$ be the free product of the $\Alt(W_j)$ acting on $\prod_i W_i$. For $j = m$ the claim is true by assumption. Assuming it is true for $j' > j$, let us prove it for $W_j$. The subgroup that only modifies the $j$th coordinate is normal, so by simplicity it is enough to show it is nonempty. Consider thus a $3$-cycle $\pi$ acting $W_j$. Its action on each $W_i$ with $i > j$ is an even permutation of order $3$. Any permutation of order $3$ on a set $W$ of cardinality at least $5$ is \emph{real (in $\Alt(W)$)}, i.e.\ conjugate to its inverse by an even permutation. Let $\pi_i \in G$ for $i > j$ act as the conjugating permutation on $W_i$ and trivially on all other $W_{i'}$ (which is possible by induction). Then defining $\gamma = \prod_i \pi_i$, the action of $\pi \circ \pi^{\gamma}$ is that of $\pi^2 \neq \ID$ on $W_j$ and trivial on $W_i$ for $i > j$, and we conclude by induction.
\end{proof}

\begin{lemma}
\label{lem:WhatToDoWhenHalts}
Suppose the $n$th program halts. Then $\alpha'(n)$ and $\beta'(n)$ are conjugate by a composition of even marker permutations, which is an involution.
\end{lemma}

\begin{proof}
Let $k$ be the radius of the clopen set $C$ defined above, and let $\ell > k$ be such that during its period $M_n$ does not travel more than $\ell$ steps. For all $\ell' \leq 3\ell$ define $W_{\ell'} \subset (B \setminus A) A^{\ell'}$ as the set of words where the subword from $A^{\ell'}$ forms a single segment in the terminology of the previous section. More precisely, there is at most one head and all $\{\leftarrow,\rightarrow\}$-arrows point toward the head, or all arrows point to the right if no head is visible. Let $W \subset W_{3\ell}$ be the subset where there is a head, which is additionally at distance at most $2\ell$ from the blinker (the symbol from $B \setminus A$ in the first coordinate), and let $U \subset W$ be the subset of words where the simulated tape of $M_n$ visible around the head (in the conveyor belt interpretation) is in $C$ and the head is at distance at most $\ell$ from the blinker. The blinker value in the first symbol from $B \setminus A$ can be arbitrary in all these sets.

Observe that due to the blinker in the first coordinate, each $W_{\ell'}$ is a mutually unbordered set, and so is $W$. We separately permute $W$ and the sets $W_{\ell'}$ with $1 \leq \ell' < 3\ell$. On $bw \in W$, $b \in B \setminus A$, our permutation fixes $w$ and changes $b$ as follows. If the configuration $bw$ is in the orbit of some $au \in U$ (in the obvious sense, noting the head does not exit the word in the application of $\alpha'(n)$), under the action of $\alpha'(n)$, then let $au \in U$ be the orbit representative. Observe that the activity bit has the same value in $a$ and $b$. Now, count the number of times the head visits the first cell of $w$ while in an active state and while on the top track, over the partial $\alpha'(n)$-orbit from $au$ to $bw$. If this number is odd, flip the activity bit, and if $t'$ is the last time step after $au$ and before $bw$, such that the head visits the first cell of $w$ while in an active state and while on the top track, and $t$ is minimal with $bw = \alpha'(n)^{t' + t}(au)$, then flip the blinker bit of $b$ if and only if $t$ is odd.

Observe that (in both $\alpha'(n)$ and $\beta'(n)$), we flip the first dummy bit whenever we flip the activity bit of a blinker, and the blinker bit returns to its value between any two active steps of the machine because the machine is only active on every fourth step. Thus, the period of $au$ is the same in the actions of both $\alpha'(n)$ and $\beta'(n)$. The conjugating map above simply records the difference in how $\alpha'(n)$ and $\beta'(n)$ modify the state of the blinker. It is easily seen to be an involution, because toggling a bit (under some condition not involving said bit) is an involution. The permutation is even since the second dummy bit of the one-head machine's state is never modified.

It is now not hard to see that the above marker CA corresponding to the involution $\pi$ on the mutually unbordered set $W$ conjugates the action of $\alpha'(n)$ to that of $\beta'(n)$ in all long enough segments where $M_n$ is simulated: If in some segment the head is further than $\ell$ steps to the right of the closest blinker to the left, or there is no blinker immediately to the left of the segment, then the two maps behave identically and the conjugating map acts trivially. If the segment is of length at least $3\ell$ and the head visits the cell immediately to the right of the blinker during the orbit, then there is a $C$-representative of the orbit at distance at most $\ell$ of the blinker, and the permutation was chosen precisely so that it conjugates the differing actions on the blinker to each other.

Next, we deal with the short segments. First suppose we were allowed to permute words in $W_{\ell'}$ under the condition that the simulating segment ends immediately to the right of this word -- these maps are automorphisms, but not marker automorphisms.\footnote{Actually, the technical result about $\CPAut$ allows such an additional condition, so we could simplify the proof a bit. However, this overlap admits a simpler statement for the present lemma.} In this case, we simply observe that even though the behavior of the machine in the simulation by $\alpha'(n)$ and $\beta'(n)$ does not necessarily correspond to the movement of $M_n$ on any infinite configuration, the logic used two paragraphs above still works, and we can again conjugate the behaviors of the blinker bits by picking an orbit representative. Let $\pi_{\ell'}$ be this even involution.\footnote{If we did not insist on the conjugating automorphism being an involution, we could simply observe that the logic from two paragraphs above shows that the cycle structures of $\alpha'(n)$ and $\beta'(n)$ on $W_{\ell'}$ are necessarily the same, so since all cycles are even, they are conjugate by an even permutation.}

Let $f$ and $f_{\ell'}$ for $\ell' < 3\ell$ be the automorphisms described above (so $f$ applies $\pi$ on each occurrence of $W$, and $f_{\ell'}$ applies $\pi_{\ell'}$ on $W_{\ell'}$ if the segment ends immediately to the right of the word). It is easy to see that $f$ and $f_{\ell'}$ commute, and thus the composition $f' = f \circ f_1 \circ \cdots f_{3\ell'-1}$ is also an involution. It remains to show that $f'$ is a composition of even marker permutations (where we do not know whether the segment ends).

For this observe that for a fixed $\ell' \leq 3\ell$ we can perform any even permutation on the set $W_{\ell'}$ (note that $W \subset W_{3\ell}$, so this includes $\pi$), as a marker automorphism. The marker automorphism does not see whether the segment continues, so if it does continue this has the side-effect that we also modify  segments of the form $W_{\ell'+j}$ for $j \geq 1$ occurring on the tape.

Thus what we have is an action of the free product of the alternating groups $\Alt(W_{\ell'})$ on $\prod_{\ell' \leq 3\ell} W_{\ell'}$ where the action of $\Alt(W_{\ell'})$ on $W_{\ell'}$ is the natural one, but there is a side-effect on the components $W_{\ell'+j}$ as described above. The previous lemma is precisely about this situation, and shows that the group generated includes the natural actions of the groups $\Alt(W_{\ell'})$ individually, and by definition we can build $f'$ from them.
\end{proof}

\subsection{Permutation engineering}

In this section we show how to perform the permutation described in the previous section inside $G$. We show that a large class of marker-type RCA (more general than marker automorphisms) can be implemented in $\CPAut_2[A;B]$. In fact, we show a more general statement that allows applying ``local permutations'' in rather general contexts.

\begin{lemma}
Let $A = \{0,1,...,n-1\}$. Let $\B = \mathcal{P}(\mathcal{P}(A))$. Let $\C$ be smallest subset of $\B$ containing $I_i = \{B \subset \{0,1,...,n-1\} \;|\; i \in B\}$ for all $i \in A$, such that
\[ C, C' \in \C \implies C \cup C' \in \C \wedge C \cap C' \in \C \wedge C \cap (C')^c \in \C. \]
Then $\C = \{C \in \B \;|\; \emptyset \notin C\}$.
\end{lemma}

\begin{proof}
Obviously $\emptyset \notin C$ for all $C \in \C$, by induction, by $\emptyset \notin I_i$ for all $i$, and the form of the closure properties. Thus, it is enough to show $\{C \in \B \;|\; \emptyset \notin C\} \subset \C$. Since we have closure under union, it is enough to show that for each $X \in \mathcal{P}(A) \setminus \emptyset$ we have $\{X\} \in \C$. If $j \in X$ then
\[ \{X\} = I_j \cap \bigcap_{i \in X} I_i \cap \bigcap_{i \notin X} I_i^c \]
which can be constructed from the latter two closure properties.
\end{proof}

The idea of the following lemma is roughly the following: Suppose that for any $g \in G$ and $a \in A$, we can act by $g$ (on something) whenever some variable has value $a$ and some condition $a \in X$ holds; while when the variable has a different value, or $a \notin X$, this action does not happen. We do not actually know what the condition $X$ is, and only know that under at least some condition on $a$ the action takes place. Then assuming we can also permute the variable value transitively, we can perform the action depending on $a$, in a uniform way that works independently of the particular condition $X$.

\begin{lemma}
\label{lem:Confusing}
Let $G$ be a perfect group, $A$ a finite set, $H$ a group acting transitively on $A$. Consider the group $K$ acting on $A \times G \times \power(A)$, with $a$ ranging over $A$, $g$ over $G$, and $\pi$ over $H$, generated by:
\begin{itemize}
\item $p_{g,a}(a', g', X) = \left\{\begin{array}{ll}
(a',gg',X) & \mbox{if } a' = a \in X, \\
(a',g',X) & \mbox{otherwise}
\end{array}\right.$
\item $q_{\pi}(a', g', X) = (\pi(a'),g',X)$.
\end{itemize}
If $|A| \geq 3$ and $|W| \geq 5$, then $K$ contains the following element for all $g \in G$ and $a \in A$:
\begin{itemize}
\item $r_{g,a}(a', g', X) = \left\{\begin{array}{ll}
(a',gg',X) & \mbox{if } X \neq \emptyset, a' = a \\
(a',g',X) & \mbox{otherwise}
\end{array}\right.$.
\end{itemize}
\end{lemma}
 
\begin{proof}
Fix $a \in A$. Let $\C \subset \power(\power(A))$ be the set of subsets $C \subset \power(A)$ such that we can condition any $G$-action on the event $X \in C$, and the $A$-value being $a$, i.e.\ the set of $C$ such that the element
\[ \gamma_{g,C}(a', g', X) = \left\{\begin{array}{ll}
(a',gg',X) & \mbox{if } X \in C \wedge a' = a, \\
(a',g',X) & \mbox{otherwise}
\end{array}\right. \]
is in the group. Clearly we have $\{X \;|\; X \ni a'\} \in \C$ for all $a' \in A$, namely this is $p_{g,a'}^{q_{\pi}}$ for any $\pi(a) = a'$. If $C, C' \in \C$ then $C \cap C' \in \C$ by $\gamma_{[g,g'],C \cap C'} = [\gamma_{g,C}, \gamma_{g',C'}]$, perfectness of $G$ and the formula $\gamma_{g \circ g',C} = \gamma_{g,C} \circ \gamma_{g',C}$; $C \cup C' \in \C$ by $\gamma_{g, C \cup C'} = \gamma_{g, C} \circ \gamma_{g, C'} \circ \gamma_{g^{-1}, C \cap C'}$; $C \cap (C')^c \in \C$ by $\gamma_{g, C \setminus C'} = \gamma_{g, C} \circ \gamma_{g^{-1}, C \cap C'}$. By the previous lemma, $\C = \{C \subset \power(A) \;|\; \emptyset \notin C\}$. In particular, $C = \power(A) \setminus \{\emptyset\} \in \C$, and clearly $\gamma_{g,C} = r_{g,a}$.
\end{proof}

\begin{lemma}
\label{lem:FineWilfTypeThing}
Let $x, y \in A^\Z$ satisfy $x \neq y, \sigma^p(x) = x, \sigma^q(y) = y$ where $0 < p < q < n$, and write $P = \{i \in \Z \;|\; x_i \neq y_i\}$. If $0 \in P$ then $|P \cap (\llb -q + 1, p \rrb)| \geq 2$.
\end{lemma}

\begin{proof}
Suppose the contrary, i.e.\ $0 \in P$ and $P$ does not intersect $\llb -q + 1, -1 \rrb \cup \llb 1, p \rrb$. We may assume by symmetry that $0 \in D$ and $D$ does not intersect $\llb 1, p \rrb$. Let $x_0 = a \neq b = y_0$, and deduce $x_p = a = a = y_p$ from the periodicity of $x$ and the assumption on $D$. Similarly deduce $y_{p-q} = a = y_{t_1}$ where $t_1 \in \llb 1, p-1 \rrb$ and $t_1 \equiv p - q \bmod q$ (using also the periodicity of $y$). Continue inductively to deduce $y_{t_i} = a$ for all $i \geq 0$, where $t_i \equiv p - iq \bmod q$. In particular we get $y_{t_{pq}} = y_0 = a$, a contradiction.
\end{proof}

We need a version of the marker lemma. The standard one would suffice, but we give a superficially stronger version.

\begin{lemma}
Let $X$ be a subshift. For any $m,n$ there exists a block map $f_M : X \to \{0,1\}^\Z$ such that
\begin{itemize}
\item the subshift $f(X)$ does not contain any words $1 0^k 1$ with $k < n-1$.
\item if $x_{\llb -m, m \rrb}$ is not periodic with period less than $n$, then $f_M(x)_{\llb -2n+1, 0 \rrb}$ contains the symbol $1$.
\end{itemize}
\end{lemma}

\begin{proof}
This is obtained from the block map version \cite[Lemma~2]{Sa12b} of the marker lemma \cite[Lemma~10.1.8]{LiMa95} by spreading the areas containing markers for $m$ steps in both directions with period $n$, unless two such areas meet.
\end{proof}

We call the $1$s output by $f_M$ \emph{markers}. Informally, the lemma says that unless the point has a small period for a long time ($m$ steps), we can use the local rule of $f_M$ to read off markers, which are always separated by a distance of at least~$n$. 


In what follows, we will be working mostly with the subshift $((A \times B)^\Z, \sigma^k)$, i.e.\ we work directly with points in $(A \times B)^\Z$, but everything is only $\sigma^k$ invariant, and automorphisms and clopen sets can depend on the position modulo $k$. This requires us to restate some basic definitions explicitly for higher power shifts; this could be avoided by working more intrinsically with $((A \times B)^\Z, \sigma^k)$, but we find that 
it is notationally clearer to keep $k$ explicit.

Consider a closed set $X \subset A^\Z$ (not necessarily shift-invariant). A \emph{local permutation} on $X$ is a homeomorphism $\chi : X \to X$ such that for some $m$, $\chi(x)_{\Z \setminus \llb -m,m \rrb} = x_{\Z \setminus \llb -m,m \rrb}$ for all $x \in X$. We never need to iterate local permutations, so we use superscripts for positional application rather than iteration: For $i \in \Z$ such that $\sigma^i(X) = X$, we use the shorthand $\chi^i = \chi^{\sigma^i}$, and we also say we \emph{apply $\chi$ at $i$}.
We write $\bar \chi$ for the inverse homeomorphism of $\chi$. 

We say a local permutation $\chi$ is \emph{$k$-self-commuting} if ($\sigma^k(X) = X$ and) $[\chi^i, \chi^j] = \ID$ for all $i, j \in k\Z$, and \emph{self-commuting} if it is $1$-self-commuting (in which case $(X, \sigma)$ is a subshift). If $\chi$ is $k$-self-commuting, for any $I \subset k\Z$ we can unambiguously write $\chi^I$ for the result of applying $\chi$ in every position $i \in I$ (along any well-ordering of $I$). The function $\chi^I$ is continuous, and when $I+t = I$ it commutes with $\sigma^t$ (see e.g.\ \cite{Sa18c}). Write $\LP(X)$ for the group of all local permutations on the closed set $X$.

A local permutation has \emph{radius\footnote{This roughly corresponds to the ``strong shift-invariant radius'' from \cite{Sa18c}.} $k$ at $j$} if $\chi(x)_{\llb j, j+k-1 \rrb}$ is only a function of $x_{\llb j, j+k-1 \rrb}$ and $\chi(x)_i = x_i$ for $i \notin \llb j, j+k-1 \rrb$, and \emph{radius $k$} if this holds for some $j \in \Z$. We say an element $\chi \in \LP(X)$ has \emph{evenradius $k$ at $j$} if $\chi = [\chi_1, \chi_2]$ where $\chi_1$ and $\chi_2$ have radius $k$ at $j$, and has \emph{evenradius $k$} if it has evenradius $k$ at some element.




A clopen set $D \subset X$ is \emph{$(k,\chi)$-ignorant} if for all $x \in X$, $i \in \Z$, we have $x \in D \iff \chi^i(x) \in D$, whenever $i \in k\Z$. If $G$ is a group of local permutations, $D$ is \emph{$(k,G)$-ignorant} if it is $(k,\chi)$-ignorant for all $\chi \in G$. If $D$ is $(k,\chi)$-ignorant and $\chi$ is $k$-self-commuting, define $\chi^{(k,D)} : B^\Z \to B^\Z$ by
\[ \chi^{(k,D)}(x) = \chi^I(x), \mbox{ where } i \in I \iff i \in k\Z \wedge \sigma^i(x) \in D. \]
It is straightforward to check that $\chi^{(k,D)}$ is an automorphism of $(B^\Z, \sigma^k)$ (with inverse $\bar \chi^{(k,D)}$).

Say a set $I \subset \Z$ is \emph{$(k,n)$-separated} if $|j-j'| \geq kn$ for any two distinct $j, j' \in I$, and $I \subset k\Z$. Say a clopen set $D \subset B^\Z$ is \emph{$(k,n)$-aperiodic} if it is $n$-aperiodic as a clopen set in $(B^\Z, \sigma^k)$, i.e.\
\[ \sigma^i(D) \cap D \neq \emptyset \wedge i \in k\Z \setminus \{0\} \implies |i| \geq kn. \]

\begin{definition}
Let $X \subset B^\Z$ satisfy $X = \sigma^k(X)$. We say a group $G \leq \LP(X)$ has property $P_{n,k}$ if it is generated by $k$-self-commuting local permutations $\chi$ such that the following properties hold:
\begin{enumerate}
\item there is a $(k,G)$-ignorant $(k,n)$-aperiodic clopen set $D$ such that $\chi^{k\Z} = \chi^{(k,D)}$ for all $\chi \in G$,
\item we can write each $\chi \in G$ as an ordered product $\chi = \prod_{i = 1}^\ell \chi_i$ where $\chi_i \in \LP(X)$ has even radius at most $k$, and for any $(k,n)$-separated $I \subset k\Z$ we have $\chi^I = \prod_{i = 1}^\ell \chi_i^I$.
\end{enumerate}
\end{definition}

Note that the even radius $k$ of $\chi_i$ is interpreted in $(X, \sigma)$ rather than $(X, \sigma^k)$, i.e.\ $\chi_i$ modifies some $k$ consecutive coordinates $\llb j, j+k-1 \rrb \subset \Z$ in configurations of $X \subset B^\Z$, where we do not necessarily have $j \in k\Z$.

\begin{definition}
\label{def:fkechi}
If $E \subset A^\Z$ is clopen and $\chi \in \LP(B^\Z)$ is $k$-self-commuting, then define $f_{k,E,\chi} \in \Aut(((A \times B)^\Z, \sigma^k))$ by
\[ f_{k,E,\chi}(x,y) = (x,\chi^I(y)) \mbox{ where } i \in I \iff i \in k\Z \wedge \sigma^i(x) \in E \]
\end{definition}

We note that if $\chi$ is self-commuting also as an element of $\LP(B^\Z)$, then
\[ f_{1,E,\chi} =  f_{k,E,\chi} \circ f_{k,\sigma(E),\chi^{-1}} \circ \cdots f_{k,\sigma^{k-1}(E),\chi^{-k+1}}. \]

\begin{lemma}
\label{lem:constru}
Suppose $G \leq \LP(B^\Z)$ is a perfect group with property $P_{n,k}$, let $E \subset A^\Z$ be any clopen set, and suppose $|A^k|,|B^k| \geq 5$. Then
\[ f_{k,E,\chi} \in \CPAut_k[A;B] \]
for all $\chi \in G$.
\end{lemma}

The conditions $|A^k|,|B^k| \geq 5$ ensure that $\Alt(A^k), \Alt(B^k)$ are themselves perfect groups. Since the lower central series of $\Sym(A^k)$ terminates in $\Alt(A^k)$ once $|A^k| \geq 3$, we could replace this by $3$, but we feel doing everything with perfect groups is less of a mental strain, and in our application we have full freedom over the cardinalities anyway.

\begin{proof}
Apply the marker lemma to $(A^\Z, \sigma^k)$ with the $n$ from the property $P_{n,k}$, and $m$ sufficiently large (determined later). As stated, the marker lemma applies to a specific coding of the subshift, but we can interpret this as follows: on any $x \in A^\Z$ the markers are a $(k,n)$-separated set (i.e.\ $kn$-separated and on positions in $k\Z$), and whenever the subword in $\llb -mn, mn-1 \rrb$ is not $(k,n')$-periodic for any $n' < n$, there is a marker in $\llb -2kn+1, 0 \rrb$. (We say $w \in A^*$ is \emph{$(k,n)$-periodic} if $w_{i+kn} \neq w_i$ for all applicable $i$.) Let $D$ be the clopen set from the definition of $P_{n,k}$.

Recall that $\CPAut_k[A;B]$ is the subgroup of $\Aut(((A \times B)^\Z, \sigma^k))$ where the generators allow applying the shift $\sigma$ on either track separately, and applying permutations simultaneously in the blocks $\llb ki,ki+k-1 \rrb$ (for all $k \in \Z$).

In particular, by conjugating with the shift $\sigma$, we can also apply permutations in blocks $\llb ki+j,ki+j+k-1 \rrb$ offset by any $j$. As a special case, we can apply permutations in blocks $\llb ki+j,ki+j+k-1 \rrb$ on the second track, under some condition on the corresponding block on the first track, and vice versa with the tracks exchanged. Conjugating with partial shifts, we can also condition permutations on one track on information elsewhere on the other track.

We assume $|A| \geq 5$ for minor notational convenience; otherwise use $A^k$ in place of $A$ in the remaining argument. First observe that it is enough to construct for each $\chi \in G$ and $b \in A$ the automorphism $f_{k,[b],\chi}$. Namely, conjugating with partial shifts we obtain $f_{k,U,\chi}$ for basic cylinders $U = [b]_i$. Since $G$ is perfect and the basic cylinders are a base of the Boolean algebra of clopen sets, all automorphisms $f_{k,E,\chi}$ can be built using the formulas
\begin{align*}
f_{k,C,\chi \circ \chi'} &= f_{k,C,\chi} \circ f_{k,C,\chi'}, \\
f_{k,C \cap C', [\chi,\chi']} &= [f_{k,C,\chi}, f_{k,C',\chi'}], \\
f_{k,C \cup C', \chi} &= f_{k,C, \chi} \circ f_{k,C', \chi} \circ f_{k,C \cap C', \chi^{-1}}, \\
f_{k,A^\Z \setminus C, \chi} &= f_{k,A^\Z, \chi} \circ f_{k,C, \chi^{-1}}.
\end{align*}

Now, fix $b \in A$ and $\chi \in G$, and let $\chi_1, \ldots, \chi_\ell$ be the local permutations from the definition of $P_{n,k}$ for $\chi$, in particular $\chi = \prod_{i = 1}^\ell \chi_i$. Consider the automorphism $g_{i,j'}$ that applies $\chi_i$ at $j \in k\Z$ whenever there is a marker on the first track at $j-j'$ but no markers in $\llb j-j'+1, j \rrb$, and the symbol at $j$ on the first track is $b$.

Then $g_{i,j'} \in \CPAut_k[A;B]$: Since $\chi_i$ has even radius $k$ and $B^\Z$ is a full shift, $\chi_i$ simply permutes the set of words in some fixed length-$k$ block by some even permutation. Since $|B^k| \geq 5$, the alternating group acting on words in that block is perfect. Since we can clearly condition an even permutation on a single symbol of the first tape by using the generators of $\CPAut_k[A;B]$, the usual commutator formulas (analogous to the previous paragraph) show we can perform any even permutation in this block, in any context read from the first track. In particular, we can read arbitrary information from the first tape and interpret it through the marker CA $f_M$ (additionally checking the symbol at $j$ is $b$), giving the automorphism $g_{i,j'}$.

By the assumption on $\chi$ and the $\chi_i$, and the fact markers have distance at least $n$, composing the maps $g_{i,j'}$ over $i$, we get the automorphism $g_{j',\chi}$ which applies $\chi$ at $j \in k\Z$ if and only if there is a marker on the first track at $j-j'$ but no markers in $\llb j-j'+1, j \rrb$, and the symbol at $j$ on the first track is $b$ (since the set of positions $j \in k\Z$ where these conditions hold is clearly $(k,n)$-separated).

Composing over $j' = 0, \ldots, 2kn-1$, we obtain the automorphism $h_{b,\chi}$ that fixes the content on the first track, and applies the local permutation $\chi$ on the second track at $j$ whenever the symbol on the first track at $j$ is $b$, and we \emph{see markers nearby}, i.e.\ there a marker on the first track at most $2kn-1$ steps to the left of $j$; and otherwise behaves as the identity.

Now we know how to build the automorphisms $h_{b,\chi}$ for $b \in A$ and $\chi \in G$. We show how to use them in conjunction with Lemma~\ref{lem:Confusing} to build the automorphisms $f_{k,[b],\chi}$. To do this, observe first that by the commutator formulas, we can apply an arbitrary even permutation on the symbol over each occurrence of $D$ in a position in $k\Z$, without affecting any other coordinates, and without modifying the bottom track. Pick $A' = \{b_0, b_1, b_2\} \subset A$, $|A'| = 3$, $b = b_0$. Then $H = \langle h \rangle \cong \Z_3$ acts simply transitively on $A'$ by $h b_j = b_{j+1}$ (indexing modulo $3$).

Consider now a fixed position $j \in k\Z$ on some configuration $(x,y)$, where $\sigma^j(y) \in D$, and $x_j = b$. Let $X \subset A'$ be the set of all $b' \in A'$ such that if we have rotated symbols above all occurrences of $D$ in positions in $k\Z$ by $h \in H$ such that $hb = b'$, then there is a marker on the first track in the interval $\llb j-2n, j \rrb$ (we need the action of $H$ to be simply transitive for this to make sense).

We claim that if $m$ was chosen large enough in the application of the marker lemma, then $X$ is nonempty (and the requirement on $m$ is uniform in the triple $(x,y,j)$). In fact, this cannot even fail for two distinct $b' \in A'$, for large enough $m$. Namely, if we can find triples $(x,y,j)$ where it fails, for arbitrarily large $m$, then by compactness there is a pair of points $x, x' \in A^\Z$ such that both $x$ and $x'$ have period less than $kn$, and $x$ differs from $x'$ in a nonempty sequence of positions $I$ which is $(k,n)$-separated (since occurrences of $D$ at positions in $k\Z$ are $(k,n)$-separated). This contradicts Lemma~\ref{lem:FineWilfTypeThing}.

Now observe that applying $\chi$ on the second track (in any positions in $k\Z$) does not affect the set of occurrences of $D$, and thus we are in the situation of Lemma~\ref{lem:Confusing}: We can apply any individual permutation $\chi \in G$ at $j \in k\Z$, assuming the symbol on the top track contains the symbol $b$, whenever there are markers nearby. For some nonempty subset $X$ of $A'$, after rotating the element above $j$ to some $b' \in X$, we do see markers nearby, and the permutation $\chi$ is performed. Lemma~\ref{lem:Confusing} then implies that for any $\chi \in G$, we can construct the automorphisms that apply the permutation independently of what this set $X$ is. This automorphism then applies $\chi$ at $j \in k\Z$ precisely when the symbol at $j$ on the first track is $b$.

As explained in the fourth paragraph of the proof, the automorphism $f_{k,E,\chi}$ can be built from these automorphisms by the commutator formulas.
\end{proof}

The following is \cite[Lemma~3]{Sa18a}.

\begin{lemma}
\label{lem:UniversalGates}
Let $A$ be a finite alphabet with $|A| \geq 3$. If $n \geq 2$, then every even permutation of $A^n$ can be decomposed into even permutations of $A^2$ applied in adjacent cells. That is, the permutations
\[ w \mapsto w_0 w_1 \cdots w_{i-1} \cdot \pi(w_i w_{i+1}) \cdot w_{i+2} \cdots w_{n-1} \]
are a generating set of $\Alt(A^n)$, where $\pi$ ranges over $\Alt(A^2)$ and $i$ ranges over $0, 1,2,...,n-2$.
\end{lemma}

\begin{lemma}
\label{lem:InvolutionConjugate}
If the $n$th program halts, then $\alpha''(n) = \beta''(n)^g$ for some involution $g \in \CPAut_2[A;B]$.
\end{lemma}


\begin{proof}
We described the automorphism $g$ in Section~\ref{sec:WhatToDoWhenHalts} conjugating $\alpha'(n)$ and $\beta'(n)$. The conjugating element is composed of even marker permutations on $B^\Z$, and since $\alpha''(n) = \ID \times \alpha'(n)$ and $\beta''(n) = \ID \times \beta'(n)$, it is enough to show that for even marker permutations $f$ on $B^\Z$, $\ID \times f$ is of the form described in the Lemma~\ref{lem:constru}. Set $E = A^\Z$. Observe that all the mutually unbordered sets involved in the construction of $g$ are of cardinality at least $5$ by the fourth item of Lemma~\ref{lem:Mn}. Let $W$ be one of these sets and suppose $W \subset B^m$.

It is enough to apply the marker permutation in even coordinates (see the remark after Definition~\ref{def:fkechi}). Let $G \leq \LP(B^\Z)$ be the copy of the alternating group on $W$ acting by even permutations on words in $W$ occurring in $\llb 0,m-1 \rrb$ (and fixing other points in $B^\Z$). Then letting $n = \lceil m/2 \rceil$, it is easy to see that $G$ has property $P_{n,2}$, by taking as $D$ the clopen set $[W]$, and writing each $\chi \in G$ according to Lemma~\ref{lem:UniversalGates} as a composition of local permutations with radius $2$ at positions $j \in \llb 0, m-2 \rrb$. The property of the $\chi_i$ required for $P_{n,2}$ holds trivially because $\chi_i^j$ and $\chi_i^{j'}$ do not modify any common coordinates when $j \neq j'$ and $j, j' \in 2n\Z$, since $2n \geq m$.
\end{proof}

\subsection{Proof of the main result}

We now prove the main result. The nontrivial implications are
``if the $n$th program halts, then $\alpha(n) = \beta(n)^g$ for some involution $g \in G$''
and
``if the $n$th program does not halt, then $(X, \beta(n))$ does not embed into $(X, \alpha(n))$''.

We have essentially already proved the first statement.

\begin{lemma}
\label{lem:InvolutionConjugate2}
If the $n$th program halts, then $\alpha(n) = \beta(n)^g$ for some involution $g \in G$.
\end{lemma}


\begin{proof}
Recall that $\alpha(n)$ and $\beta(n)$ were built inside $G$ as images of $\alpha''(n)$ and $\beta''(n)$ inside a larger subgroup $H$ isomorphic to
\[ H' = \langle \alpha''(n), \beta''(n), T \rangle \]
where $T$ is the generating set of $\CPAut_2[A;B]$. If the $n$th program halts, then $\alpha''(n) = \beta''(n)^g$ for some involution $g \in \CPAut_2[A;B]$ by Lemma~\ref{lem:InvolutionConjugate}. The group $H'$ contains $\alpha''(n)$, $\beta''(n)$ and the generators of $\CPAut_2[A;B]$, so $\alpha''(n)$ and $\beta''(n)$ are conjugate in $H'$. Therefore $\alpha(n)$ and $\beta(n)$ are conjugate in $H$, thus also in $G$. 
\end{proof}

Now we consider the case where the $n$th program does not halt. We show that in this case, a particular type of subsystem can be found in $(X, \beta(n))$ but not in $(X, \alpha(n))$.

\begin{lemma}
\label{lem:b}
If the $n$th program does not halt, then $(X, \beta(n))$ contains a point~$x$ whose orbit closure contains both fixed points and points of period two. The system $(X, \alpha(n))$ does not.
\end{lemma}

\begin{proof}
First, we prove that $(Y, \beta'(n))$ contains a point $x$ whose orbit closure contains both fixed points and points of period two. Consider a configuration with two infinite conveyor belts, with a single inactive blinker in between, and pick the content on the rightmost conveyor belt to be $z \in A^\N$ such that a head is present, and in the simulated configuration the machine $M_n$ visits coordinate $0$ (the one simulated by the cell on the top track, immediately to the right of the blinker) an odd number of times with long gaps in between, as given by the third property of Lemma~\ref{lem:Mn}. If the head stays away from the origin for $t$ steps before and after the visits, it must travel at least $\Omega(\log(t))$ steps to the right, by a pigeonhole argument (otherwise, the configuration would be periodic).

Let $c_1(t)$ and $c_2(t)$ be the configurations at the time steps where the head is at distance $\Omega(\log(t))$ from the origin before and after a visit of oddly many steps. The central parts of $c_i(t)$ are fixed or $2$-periodic depending on the activity bit of the blinker. Since the blinker changes its activity when the machine visits the origin an odd number of times, as limit points of the $c_i(t)$ as $n \longrightarrow \infty$ we obtain both fixed points and $2$-periodic points.

Next we prove that the system $(Y, \alpha'(n))$ does not contain such points.
To see this, consider an arbitrary configuration $z \in B^\Z$. At any time where $\alpha(n)^t(z)$ is close to a fixed point, every head (if any are present) must be far from the origin (due to the mod-$4$ counter in their state, which is continually incremented), and all blinkers (if any are present) near the origin must be in an inactive state. Similarly, when we are close to a $2$-periodic point, heads are far form the origin, and there is at least one active blinker near the origin. Since the rule never modifies the positions of the blinkers, or their activity bits, this is clearly impossible.

Since the embeddings are explicit, the same proofs apply to $(X, \beta(n))$ and $(X, \alpha(n))$, interpreted on the simulated tapes.
\end{proof}

\begin{proof}[Proof of Theorem~\ref{thm:Main}]
The implication (1) $\implies$ (2) is Lemma~\ref{lem:InvolutionConjugate2}, (2) $\implies$ (3) is trivial. For (3) $\implies$ (4), suppose $\alpha(n) = \beta(n)^g$ for $g \in \Aut(X)$. Then $x \mapsto g(x)$ is an isomorphism between the $\Z^2$-systems $(X, \sigma, \alpha(n))$ and $(X, \sigma, \beta(n))$: it is a homeomorphism and $\sigma(g(x)) = g(\sigma(x))$ since $g$ is an automorphism. By assumption we have
$\alpha(n)(x) = (g \circ \beta(n) \circ g^{-1})(x)$, which implies $\alpha(n)(g(x)) = g(\beta(n)(x))$. The contrapositive of implication (6) $\implies$ (1) is Lemma~\ref{lem:b}. Thus (1) $\implies$ (2) $\implies$ (3) $\implies$ (4) $\implies$ (5) $\implies$ (6) $\implies$ (1). The equivalence of (1) and (7) is clear from the construction.
\end{proof}

\section{Conjugacy for $2V$ and Turing machines}

We now use a similar technique to prove that the conjugacy problem is undecidable in the Brin-Thompson $2V$ \cite{Br04a}. We refer to \cite{Br04a,BeBl14} for the definition of this group. Briefly, it is the subgroup of the homeomorphism group of $(A^\N)^2$ where the action of an element replaces a pair of prefixes (from a finite clopen partition) by another pair of prefixes (not necessarily of the same lengths). It can naturally be seen as a subgroup of the homeomorphism group of $A^\Z$ by writing the two copies of $A^\N$ back-to-back. Recall that this group is finitely-generated (even finitely-presented).

Our construction follows the same general lines as the proof of the main theorem. The $2V$ elements $\alpha(n)$ and $\beta(n)$ simulate the same one-head machine by elements of $2V$, and we also have some blinker symbols on the tape. The difference between $\alpha(n)$ and $\beta(n)$ is that in the case of $\beta(n)$ the blinker bit is flipped when a head touches it (and of course blinkers cannot change their state when the ``head'' is not nearby).

A small additional (mainly expositional) complication in the argument is that we cannot use a prefix-coding to turn moving-tape configurations into $2V$-configurations in a bijective manner (this is what is done in \cite{BeBl14}). Namely, if we did, conjugation might destroy the interpretation and make the argument more difficult. Thus we instead act on concatenations of elements of an unbordered set of words. This necessarily leads to partial tapes, and we thus apply our $2V$ elements on conveyor belts.

The elements $\alpha(n)$ and $\beta(n)$, as well as the conjugating elements, are in fact elements of the group of reversible Turing machines $\mathrm{RTM}(2,1) \leq 2V$ defined in \cite{BaKaSa16}, and the groups $\mathrm{RTM}(a,b)$ with $a \geq 2$ can be dealt with after minor modifications. We refer to \cite{BaKaSa16} for the definition of these groups.

\begin{theorem}
\label{thm:2V}
Let $G$ be the Brin-Thompson group $2V$ or any of the Turing machine groups from \cite{BaKaSa16} with nontrivial tape alphabet. There exist computable functions $\alpha, \beta : \N \to G$ such that the following are equivalent:
\begin{itemize}
\item the $n$th problem halts,
\item $\alpha(n)$ and $\beta(n)$ are of finite order,
\item $\alpha(n)$ and $\beta(n)$ are conjugate by an involution,
\item $\alpha(n)$ and $\beta(n)$ are conjugate.
\end{itemize}
In particular, $G$ has undecidable conjugacy problem.
\end{theorem}

\begin{proof}
We give the proof for $2V$. Consider the one-head machine $M_n$ from Lemma~\ref{lem:Mn} over state set $Q$ and tape alphabet $A$. We need slightly fewer technical properties: the counter modulo $4$ is not needed, and we need only one dummy bit. (The third item, describing a detail about the global behavior, is needed, but does not necessitate any modifications to the construction.)

Let $W \subset \{0,1\}^n$ be a mutually unbordered set of cardinality $|A^2| + 2$, let $V \subset W$ be a subset of cardinality $|A^2|$, and let $U \subset \{0,1\}^n$ be a mutually unbordered set of cardinality $2|Q \times A|$. Our rule is the following: if the word to the right of the origin is not in $U$, we do nothing. If it is in $U$, we read the coding of a conveyor belt from the tape: Interpret the $U$-word at $\llb 0,n-1 \rrb$ as an element of $(Q \times A) \cup (A \times Q)$, giving a state and a tape symbol. Whether the word encodes an element of $Q \times A$ or $A \times Q$ indicates whether the head is on the top or bottom track. Now read a maximal number of $V$-words around this word, i.e.\ in the contents of $(-\infty, -1\rrb$ and $\llb n, \infty)$, and interpret them in the standard conveyor belt fashion: if $x_{\llb n, \infty)} \in V^\omega$, interpret it as two tapes over alphabet $A$, the bottom one being inverted, and if the maximal $V^*$-progression is finite, close it up into a conveyor belt; on the left mutatis mutandis. 

Now, $\alpha(n)$ is the $2V$-element that applies $M_n$ through this interpretation, and, letting $W \setminus V = \{w, w'\}$, if $x_{\llb 0,n-1 \rrb} \in Q \times A$ (through the interpretation) and $x_{\llb-n,-1 \rrb} \in W \setminus V$, additionally flips the dummy bit of the state of $M_n$.

The element $\beta(n)$ does the same, but additionally applies the involution $w \leftrightarrow w'$ in $x_{\llb-n,-1\rrb}$, i.e.\ interpreting elements of $W \setminus V$ as blinkers with blinker bit $w \sim 0, w' \sim 1$, the blinker bit is flipped if (through the interpretation) the head visits the cell to the right of it, on the top track.

It is easy to see that $\alpha(n)$ and $\beta(n)$ are periodic if and only if $M_n$ is (and otherwise neither one is): If $M_n$ is not periodic then $\alpha(n)$ and $\beta(n)$ have arbitrarily large orbits by considering configurations of the form $V^{-\omega} . U V^{\omega}$ encoding legal configurations on the top track. If $M_n$ is periodic, this bounds the period on all large enough simulated tracks, and on smaller ones there is a bound on the period by the pigeonhole principle.

Now observe that $\alpha(n)$ and $\beta(n)$ are conjugate by an involution in $2V$ when $M_n$ halts: as in the proof of Lemma~\ref{lem:WhatToDoWhenHalts}, the conjugating element simply adjusts the value of the blinker bits (based on some clopen choice of orbit representatives as in Section~\ref{sec:WhatToDoWhenHalts}). Because of the dummy bit in the state of the machine, the periods of $\alpha(n)$ and $\beta(n)$ are the same, and this process is well-defined.

If the $n$th program does not halt, then using the third item of Lemma~\ref{lem:Mn} we can find configurations $x \in A^\Z$ where the blinker bit is far to the left, and in the iteration of $\beta(n)$ we go flip it, and come back. More precisely, for some $x, x', t$ we have $\beta(n)^{t}(x) = x'$ such that exactly one blinker occurs in $x$ (resp.\ $x'$) at $-i$ (resp.\ $-j$), for some large $i, j$, and we have
\[ x_{\llb -i,-i+n-1 \rrb} = w \neq w' = x'_{\llb-j,-j+n-1\rrb}. \]

Now, suppose $\beta(n) = \alpha(n)^\gamma$. Then we have
\begin{equation} \beta(n)^{t}(x) = (\alpha(n)^{\gamma})^{t}(x) = \gamma^{-1} \circ \alpha(n)^{t} \circ \gamma(x). \label{eq:tconju} \end{equation}
Thinking about the leftmost blinker (the leftmost word in the set $\{w, w'\}$) of intermediate configurations, this easily gives a contradiction: If $i$ is large enough, then the first application of $\gamma$ cannot flip the leftmost blinker bit (which is at $\llb -i,-i+n-1 \rrb$), and can only translate it by a bounded amount.

Since the leftmost blinker in $\beta(n)^t(x)$ is in state $w'$, the only way we can have $(\alpha(n)^{\gamma})^{t}(x) = \beta(n)^{t}(x)$ is that the leftmost blinker in $\alpha(n)^{t} \circ \gamma(x)$ is at a bounded distance from the origin. Taking $j$ large, we conclude that the leftmost blinker in $\gamma^{-1} \circ \alpha(n)^{t} \circ \gamma(x)$ may not simultaneously be in $\llb-j,-j+n-1\rrb$, and be in the correct state $w'$. This contradicts the existence of $\gamma$.
\end{proof}

\section{The word problem}

In this section, we pinpoint the difficulty of the word problem for f.g.\ subgroups of automorphism groups of full shifts, since this complexity-theoretic statement does not seem to appear anywhere, and because the word problem is the first problem on Dehn's list whose second item our paper concerns. In \cite{GuJeKaVa18}, a similar method is used to show that the word problem of a finitely-generated subgroup of the automorphism group of a two-dimensional subshift of finite type can be undecidable.

It was proved in \cite{BoLiRu88} that the word problem is decidable for f.g.\ subgroups of $\Aut(A^\Z)$, and in \cite{KiRo90} a more quantitative statement was given, from which a complexity-theoretic upper bound can be extracted. To our knowledge, neither the co-NP upper bound nor completeness appear in the literature. 

A \emph{partial word} is a word $w \in (A \cup \{\bla\})^n$ for some $n$. Let $u \in A^n$ be a word. We say \emph{$w$ matches $u$ somewhere} if $\exists i: w_i = u_i$. In the following complexity-theoretic problem, sets of words are given in uncompressed format, as a list.

\begin{lemma}
\label{lem:NPcomplete}
Let $A$ be a fixed alphabet with $|A| \geq 2$. Given a set of partial words $W \subset (A \cup \{\bla\})^n$, it is it is NP-complete to determine whether there is a word $u \in A^n$ such that every $w \in W$ matches $u$ somewhere.
\end{lemma}

\begin{proof}
The problem is clearly in NP, indeed one can verify that a word $u$ satisfies the partial words linear time. For completeness, one can easily program $3$-SAT by having the positions of $u$, starting indexing with zero, represent variables (some $0 \in A$ represents $\bot$, and some $1 \in A$ represents $\top$). The clause $q_0 x_{j_0} \vee q_1 x_{j_1} \vee \cdots \vee q_k x_{j_k}$ where the $x_i$ are variables and the $q_i \in \{-, +\}$ representing a positive or negative term, is represented by the partial word
\[ \bla^{j_0} b_0 \bla^{j_1-j_0-1} b_2 \cdots \bla^{j_k-j_{k-1}-1} b_k \bla^{n-j_k-1} \]
where $b_i = 1 \iff q_{j_i} = +$.
\end{proof}

\begin{theorem}
\label{thm:WP}
Let $\mathcal{D}$ be the class of (locally finite)-by-$\Z$ groups. For any nontrivial alphabet $A$, the word problem of the group $\Aut(A^\Z)$ is $\mathcal{D}$-eventually locally co-NP-complete.
\end{theorem}

Here, we use similar terminology as with the conjugacy problem. The meaning should be clear: the word problem is co-NP-complete ``starting from'' some f.g.\ (locally finite)-by-$\Z$ group (in particular, it is always in co-NP).

\begin{proof}
To see that the word problem of any finitely-generated group in $\Aut(A^\Z)$ is in co-NP, consider $g = g_1 \circ g_2 \circ \cdots \circ g_n$ over the generators, and let $R$ be the maximal radius of the local rules of the generators. Then $g$ admits a local rule with radius at most $Rn$. To prove $g$ does not represent the identity, guess a word of length $2Rn+1$ on which its local rule acts nontrivially, and test this by evaluating the local rules of the $g_i$. This proves that not-representing-the-identity is in NP. This is just the complement of the word problem, which is then in co-NP.

For (eventual) co-NP-hardness, pick the alphabet to be $A \times B$ where $A = \{0,1\}$ and $|B| \geq 5$. As generators of our group $G$ take the partial shift $\sigma_1$ on the first track, and for each even permutation $\pi \in \Alt(B)$ and $a \in A$ take the automorphism $f_{a,\pi}$ defined by
\[ f_{a,\pi}(x,y)_i = \left\{\begin{array}{ll}
(x_i, \pi(y_i)) & \mbox{if } x_i = a, \\
(x_i, y_i) & \mbox{otherwise.}\end{array}\right. \]
This group is (locally finite)-by-$\Z$: there is a homomorphism to $\Z$ defined by $\pi(\sigma_2) = 1$, $\pi(f_{a,\pi}) = 0$. Every element in the kernel simply permutes the second track based on nearby contents of the top track, leading to local finiteness.

We now mimic the proof of Barrington's theorem \cite{Ba89}: To each $U \subset \{0,1\}^n$ and $\pi \in \Alt(B)$, associate the automorphism $f_{U,\pi}$ defined by
\[ f_{U,\pi}(x,y)_i = \left\{\begin{array}{ll}
(x_i, \pi(y_i)) & \mbox{if } x_{\llb i,i+n-1 \rrb} \in U, \\
(x_i, y_i) & \mbox{otherwise.}\end{array}\right. \]
For the cylinder set $U_{i,a} = \{u \;|\; u_i = a\} \subset \{0,1\}^n$, we have $f_{U,\pi} = f_{a,\pi}^{\sigma_2^i}$, thus $f_{U,\pi}$. Now, given a set of partial words $W \subset (A \cup \{\bla\})^n$, using De Morgan's laws, the identity
\[ [f_{U,\pi}, f_{V,\pi'}] = f_{U \cap V, [\pi, \pi']} \]
and the fact every element of $\Alt(B)$ is a commutator, a short calculation shows that $f_{W, \pi}$ for any permutation $\pi \in \Alt(B)$ is in $G$ and has polynomial word norm with respect to the length of $W$. Clearly if $\pi$ is nontrivial, $f_{W,\pi}$ is a positive instance for the word problem of $G$ (i.e.\ is the identity map) if and only if $W$ is a negative instance for the problem in Lemma~\ref{lem:NPcomplete}. Thus, the word problem of $G$ is co-NP-complete.
\end{proof}

We make the (mainly terminologically) interesting remark that, if $\mathrm{P} \subsetneq \text{co-NP} \subsetneq \mathrm{PSPACE}$, then the classes of automatic groups, cellular automata groups and automata groups are (eventually) separated from each other by the complexity of their word problems, in light of the above theorem, \cite{Ep92,DARoWa17}.

Topological full groups of full shifts embed in their automorphism groups (but not vice versa, as can be deduced from \cite{BaKaSa16}). The above construction can be naturally seen as a subgroup of the topological full group.

\begin{theorem}
\label{thm:topofullgroup}
Let $\mathcal{D}$ be the class of (locally finite)-by-$\Z$ groups. If $C$ is a nontrivial finite alphabet, the topological full group of $C^\Z$ has $\mathcal{D}$-eventually locally co-NP-complete word problem. 
\end{theorem}

\begin{proof}
It is well-known that topological full group elements can be represented by local rules, and it can be seen that their word problem is in co-NP analogously to the previous theorem, by simply applying a composition of local rules to the co-NP witness (a word where the rule applies nontrivially).

For the co-NP-completeness, we recall the notion of a (generalized) reversible finite-state machine from \cite{BaKaSa16}, which generalizes the topological full group. Letting $A = \{0,1,...,n-1\}$, $B = \{0,1,...,k-1\}$, the group of finite-state machines $\RFA(\Z,n,k)$ was defined in \cite{BaKaSa16} as the set of self-homeomorphisms of $f : A^\Z \times B \to A^\Z \times B$ such that for some clopen set $D \subset A^\Z \times B$ and continuous map $c : D \to \Z \times B$, whenever $\i(x,b) = (n,b')$ we have $f(x, b) = (\sigma^n(x),b')$. The interpretation is that $B$ is the \emph{set of states} of the finite-state machine, $A$ is the \emph{tape alphabet}, and the function $c$ gives the movement of the ``head'' and the state transitions.

The group defined in the previous theorem is clearly isomorphic to a subgroup of $\RFA(\Z,2,5)$, translating permutations of $B$ into permutations of the state and $\sigma_1$ into movement of the head. It is easy to see that $\RFA(\Z,n,k)$ embeds in the topological full group of any full shift $C^\Z$: Again let $|A| = n, |B| = k$, and let $U \subset C^\ell$ be a set of mutually unbordered set of words with $\ell \geq 2|B|$ and $|U| = |A^2|$.

Now, interpret words in $U$ as coding two symbols over $A$, and allocate the first $2|B|$ positions in these words for coding the current state. As always, interpret maximal runs of words in $U$ as conveyor belts. It is easy to see how to translate $c$ into the local rule of a topological full group element; when the head moves, we move along the conveyor belt, and when the state changes, we change the position inside the current $U$-word. (In other situations, we do nothing.)
\end{proof}

The above proof works directly also for the larger group of reversible Turing machines defined in \cite{BaKaSa16}, showing its f.g.\ subgroups have co-NP word problems (and the group constructed in the previous theorem has a co-NP-complete word problem).

\section{Questions}

Our main theorem is not based on entropy, and thus the computability of entropy stays open.

\begin{question}
Given $f,g \in \Aut(A^\Z)$, is it decidable whether $h_{\mathrm{top}}(f) = h_{\mathrm{top}}(g)$? Is $h_{\mathrm{top}}(f)$ computable? Is it computable from $f$? Is there a simple characterization of the set of entropies? Can an entropy separation be included in Theorem~\ref{thm:Main}?
\end{question}

One motivation for the last subquestion is that such a variant of Theorem~\ref{thm:Main} would also prove that it is undecidable whether one RCA is a factor of another (as a dynamical system), which seems difficult to obtain from our method.

The partial shifts and symbol permutations (Figure~\ref{fig:Example}, Corollary~\ref{cor:Partitioned}) provide a finitely-generated group $G$ which is very easy to define through its action, with undecidable conjugacy problem. This group has word problem in co-NP, so it is not too complicated. However, it is not so clear what subset of $\Aut(A^\Z)$ the group $G$ actually covers.

\begin{question}
Is there a f.g.\ explicitly f.g.-universal group $G \leq \Aut(A^\Z)$ whose membership problem is decidable? (= given $f \in \Aut(A^\Z)$, it is decidable whether $f \in G$.) Is membership decidable for $\PAut[A;B]$ for nontrivial alphabets $A,B$ with $\max(|A|,|B| \geq 3$?
\end{question}

When $|A| = |B| = 2$, $\PAut[A;B] \cong \Z_2^2 \rtimes \mathrm{GL}(2, \Z_2[\textbf{x}, \textbf{x}^{-1}])$ by a natural action \cite{Sa18a}, and from this it is clear that membership in $\PAut[A;B]$ is decidable. However, this group is not even f.g.-universal (because it is a linear group), let alone explicitly so.

It can be shown that $F_2 * F_2 \leq \Aut(A^\N)$ at least for large enough $A$, i.e.\ thus there are finitely-generated groups of one-sided cellular automata with undecidable conjugacy problem.

\begin{question}
Let $|A| \geq 3$. Is conjugacy decidable in $\Aut(A^\Z)$? Does $\Aut(A^\N)$ have eventually locally undecidable conjugacy? Is dynamical conjugacy (conjugacy over $\Homeo(A^\N)$) decidable in this group? What about $\Aut(X)$ for other one-sided subshifts $X$?
\end{question}

The asynchronous rational group has eventually locally undecidable conjugacy problem, as these can simulate automorphisms two-side cellular automata. Synchronous automata, can only simulate one-sided cellular automata, so we do not obtain results for them. Let $A$ be a finite alphabet.

\begin{question}
Let $A$ be a finite alphabet, and let $G$ be the group of all finite-state transductions (the join of all automata groups) on $A^\N$. Is conjugacy in $G$ decidable? Does $G$ have eventually locally undecidable conjugacy problem? Is dynamical conjugacy decidable in $G$?
\end{question}

Our proof of undecidability of conjugacy in the case of $2V$ and Turing machines is not dynamical; while it seems unlikely that the homeomorphisms $\alpha(n)$ and $\beta(n)$ built in the proof are conjugate in $\Homeo(\{0,1\}^\Z)$, the proof does not show this, and the heads could in principle accidentally emulate the ``blinkers'' by some data they are allowed to modify. It seems likely that this possibility can be ruled out by taking a more careful look at properties of \cite{KaOl08} (and possibly modifying the way the blinkers work), but this is beyond the scope of the paper.

\begin{conjecture}
Let $C$ be Cantor space and $G \leq \Homeo(C)$ be either $2V$, or the TMH or TMT action of one of the groups of Turing machines from \cite{BaKaSa16}. Then $G$ has a finitely-generated subgroup $H \leq G$ such that for pairs $(f,g) \in H^2$, conjugacy in $H$ is recursively inseparable from non-conjugacy in $\Homeo(C)$.
\end{conjecture}

Since $G = 2V$, one should take $H = G$ in this case, but the groups of Turing machines are not finitely-generated \cite{BaKaSa16}. For Turing machine groups, TMH and TMT are two different ways of seeing the group as a group of homeomorphisms of Cantor space, and while the groups themselves are isomorphic \cite{BaKaSa16}, conjugacy in $\Homeo(C)$ (i.e.\ conjugacy as dynamical systems) is different for these models.

\begin{example}
Consider Turing machines in the sense of \cite{BaKaSa16} with one state and binary alphabet. In TMT, these are elements of $\Homeo(\{0,1\}^\Z)$. Define $f,g \in \Homeo(\{0,1\}^\Z)$ by
\[ f(x)_i = \sigma^2(x)_i = x_{i+2} \;\; \mbox{ and } \;\; g(x)_i = \left\{\begin{array}{ll}
x_i & \mbox{if } i = 0 \\
x_{i+1} & \mbox{if } i \notin \{-1,0\} \\
x_{i+2} & \mbox{if } i = -1
\end{array}\right. . \]
These are not conjugate, as they have different entropies as dynamical systems. In TMT, tape movement is translated into movement of the head, and we obtain corresponding homeomorphisms $f',g' \in \Homeo(X)$ where $X = (\{0,1\}^\Z \times \Z) \cup \{0,1\}^\Z$ under a Cantor topology. It is easy to see they are conjugate, by encoding the parity of the head position into the bit carried by the head. \qee
\end{example}

One can of course ask the above question also for reversible Turing machines over other subshifts. Both $2V$ and groups of reversible Turing machines (in TMT and TMH) have obvious generalizations to arbitrary subshifts; in the case of reversible Turing machines, also over arbitrary groups.

The word problem is quite generally decidable for topological full groups.\footnote{It is decidable for any finitely-generated subgroup of the topological full group of a subshift with decidable language over a group with decidable word problem, such that the germ stabilizers of the shift action are trivial (meaning $\forall g \in G: \forall U \mbox{ nonempty open}: g|_U \neq \ID|_U$.} The torsion problem is decidable for one-dimensional sofic shifts, but is undecidable for two-dimensional full shifts \cite{BaKaSa16}.

\begin{question}
Given two elements of the topological full group of a one-dimensional full shift, is it decidable whether they are conjugate? What about other acting groups? What about other subshifts?
\end{question}

Finally, we have shown that the homeomorphism group of the Cantor space has eventually locally undecidable conjugacy problem. However, this is a rather complicated group, for example it is easy to see that it contains a copy of every countable group.

\begin{question}
For which topological spaces $X$ does $\Homeo(X)$ have eventually locally undecidable conjugacy problem? For manifolds, what about $\mathrm{Diff}(X)$ and its variants?
\end{question}

\section*{Acknowledgements}

I thank Carl-Fredrik Nyberg Brodda for useful literature pointers.

\bibliographystyle{plain}
\bibliography{../../../bib/bib}{}

\def\ocirc#1{\ifmmode\setbox0=\hbox{$#1$}\dimen0=\ht0 \advance\dimen0
  by1pt\rlap{\hbox to\wd0{\hss\raise\dimen0
  \hbox{\hskip.2em$\scriptscriptstyle\circ$}\hss}}#1\else {\accent"17
  #1}\fi}\def\cprime{$'$}
\begin{thebibliography}{10}

\bibitem{AaLe74}
St\ocirc{a}l~O. Aanderaa and Harry~R. Lewis.
\newblock Linear sampling and the $\forall \exists \forall$ case of the
  decision problem.
\newblock {\em The Journal of Symbolic Logic}, 39:519--548, 9 1974.

\bibitem{Al88}
Roger~C. Alperin.
\newblock Free products as automorphisms of a shift of finite type.
\newblock 1988.

\bibitem{AnBuMa15}
Yago Antol{\'\i}n, Jos{\'e} Burillo, Armando Martino, et~al.
\newblock Conjugacy in houghton's groups.
\newblock {\em Publicacions matem{\`a}tiques}, 59(1):3--16, 2015.

\bibitem{ArBa09}
Sanjeev Arora and Boaz Barak.
\newblock {\em Computational complexity: a modern approach}.
\newblock Cambridge University Press, 2009.

\bibitem{BaKaSa16}
Sebasti{\'a}n Barbieri, Jarkko Kari, and Ville Salo.
\newblock {\em The Group of Reversible Turing Machines}, pages 49--62.
\newblock Springer International Publishing, Cham, 2016.

\bibitem{Ba89}
David~A. Barrington.
\newblock Bounded-width polynomial-size branching programs recognize exactly
  those languages in {NC1}.
\newblock {\em Journal of Computer and System Sciences}, 38(1):150 -- 164,
  2989.

\bibitem{BeBl14}
J.~{Belk} and C.~{Bleak}.
\newblock {Some undecidability results for asynchronous transducers and the
  Brin-Thompson group 2V}.
\newblock {\em ArXiv e-prints}, May 2014.

\bibitem{Bi20}
J.~C. Birget.
\newblock The word problem of the brin-thompson group is conp-complete.
\newblock {\em Journal of Algebra}, 553:268 -- 318, 2020.

\bibitem{Bi03}
Jean-Camille {Birget}.
\newblock {Circuits, coNP-completeness, and the groups of Richard Thompson}.
\newblock {\em arXiv Mathematics e-prints}, page math/0310335, October 2003.

\bibitem{Bo57}
William~W. Boone.
\newblock Certain simple, unsolvable problems of group theory vi.
\newblock {\em Indagationes Mathematicae}, 60:227--232, 1957.

\bibitem{BoLiRu88}
Mike Boyle, Douglas Lind, and Daniel Rudolph.
\newblock The automorphism group of a shift of finite type.
\newblock {\em Transactions of the American Mathematical Society}, 306(1):pp.
  71--114, 1988.

\bibitem{Br04a}
Matthew~G. Brin.
\newblock Higher dimensional {T}hompson groups.
\newblock {\em Geometriae Dedicata}, 108(1):163--192, 2004.

\bibitem{Br93}
Ezra Brown.
\newblock Periodic seeded arrays and automorphisms of the shift.
\newblock {\em Transactions of the American Mathematical Society},
  339(1):141--161, 1993.

\bibitem{Co86}
Donald~J. Collins.
\newblock A simple presentation of a group with unsolvable word problem.
\newblock {\em Illinois Journal of Mathematics}, 30(2):230--234, 1986.

\bibitem{CoYa14}
E.~{Coven} and R.~{Yassawi}.
\newblock {Endomorphisms and automorphisms of minimal symbolic systems with
  sublinear complexity}.
\newblock {\em ArXiv e-prints}, November 2014.
\newblock Available at \url{https://arxiv.org/abs/1412.0080}.

\bibitem{CrGoWi09}
John Crisp, Eddy Godelle, and Bert Wiest.
\newblock The conjugacy problem in right-angled artin groups and their
  subgroups.
\newblock {\em Journal of topology}, 2(3):442--460, 2009.

\bibitem{CyFrKrPe18}
Van Cyr, John Franks, Bryna Kra, and Samuel Petite.
\newblock Distortion and the automorphism group of a shift.
\newblock {\em Journal of Modern Dynamics}, 13(1):147, 2018.

\bibitem{CyKr16b}
Van Cyr and Bryna Kra.
\newblock The automorphism group of a minimal shift of stretched exponential
  growth.
\newblock {\em Journal of Modern Dynamics}, 10:483--495, 2016.

\bibitem{CyKr16a}
Van Cyr and Bryna Kra.
\newblock The automorphism group of a shift of subquadratic growth.
\newblock {\em Proceedings of the American Mathematical Society},
  144(2):613--621, 2016.

\bibitem{DARoWa17}
Daniele D'Angeli, Emanuele Rodaro, and Jan~Philipp W{\"a}chter.
\newblock On the complexity of the word problem for automaton semigroups and
  automaton groups.
\newblock {\em Advances in Applied Mathematics}, 90:160--187, 2017.

\bibitem{De11}
Max Dehn.
\newblock {\"U}ber unendliche diskontinuierliche gruppen.
\newblock {\em Mathematische Annalen}, 71(1):116--144, 1911.

\bibitem{DoDuMaPe16}
Sebastian Donoso, Fabien Durand, Alejandro Maass, and Samuel Petite.
\newblock On automorphism groups of low complexity subshifts.
\newblock {\em Ergodic Theory and Dynamical Systems}, 36(01):64--95, 2016.

\bibitem{DoDuMaPe17}
Sebastian Donoso, Fabien Durand, Alejandro Maass, and Samuel Petite.
\newblock On automorphism groups of toeplitz subshifts.
\newblock {\em Discrete Analysis}, 19, 2017.

\bibitem{Ep17}
Jeremias Epperlein.
\newblock {\em Topological Conjugacies Between Cellular Automata}.
\newblock PhD thesis, Fakult\"at Mathematik und Naturwissenschaften der
  Technischen Universit\"at Dresden, 2017.

\bibitem{Ep92}
David Epstein et~al.
\newblock {\em Word processing in groups}.
\newblock CRC Press, 1992.

\bibitem{EpHo06}
David Epstein and Derek Holt.
\newblock The linearity of the conjugacy problem in word-hyperbolic groups.
\newblock {\em International Journal of Algebra and Computation},
  16(02):287--305, 2006.

\bibitem{FrScTa19}
Joshua Frisch, Tomer Schlank, and Omer Tamuz.
\newblock Normal amenable subgroups of the automorphism group of the full
  shift.
\newblock {\em Ergodic Theory Dynam. Systems}, 39(5):1290--1298, 2019.

\bibitem{GrNe00}
Rostislav~Ivanovich Grigorchuk and VV~Nekrashevich.
\newblock The group of asynchronous automata and rational homeomorphisms of the
  {C}antor set.
\newblock {\em Mathematical Notes}, 67(5):577--581, 2000.

\bibitem{GuSa97}
Victor Guba and Mark Sapir.
\newblock {\em Diagram groups}, volume 620.
\newblock American Mathematical Soc., 1997.

\bibitem{GuJeKaVa18}
Pierre Guillon, Emmanuel Jeandel, Jarkko Kari, and Pascal Vanier.
\newblock Undecidable word problem in subshift automorphism groups.
\newblock {\em CoRR}, abs/1808.09194, 2018.

\bibitem{GuSa17}
Pierre Guillon and Ville Salo.
\newblock {\em Distortion in One-Head Machines and Cellular Automata}, pages
  120--138.
\newblock Springer International Publishing, Cham, 2017.

\bibitem{GuZi13}
Pierre Guillon and Charalampos Zinoviadis.
\newblock Densities and entropies in cellular automata.
\newblock In S.~Barry Cooper, Anuj Dawar, and Benedict L\"owe, editors, {\em
  How the world computes}, pages 253--263. Springer, 2013.

\bibitem{HaKrSc20}
Yair Hartman, Bryna Kra, and Scott Schmieding.
\newblock The stabilized automorphism group of a subshift, 2020.

\bibitem{Hi74}
Graham Higman.
\newblock {\em Finitely presented infinite simple groups}, volume~8.
\newblock Dept. of Pure Mathematics, Dept. of Mathematics, IAS, Australian
  National~…, 1974.

\bibitem{Ho10}
Michael Hochman.
\newblock On the automorphism groups of multidimensional shifts of finite type.
\newblock {\em Ergodic Theory Dynam. Systems}, 30(3):809--840, 2010.

\bibitem{Ho66}
Philip~K. Hooper.
\newblock The undecidability of the turing machine immortality problem 1.
\newblock {\em The Journal of Symbolic Logic}, 31(2):219--234, 1966.

\bibitem{HuKaCu92}
Lyman~P. Hurd, Jarkko Kari, and Karel Culik.
\newblock The topological entropy of cellular automata is uncomputable.
\newblock {\em Ergodic Theory Dynam. Systems}, 12(2):255--265, 1992.

\bibitem{JaKa20}
Joonatan Jalonen and Jarkko Kari.
\newblock On the conjugacy problem of cellular automata.
\newblock {\em Information and Computation}, 274:104531, 2020.
\newblock AUTOMATA 2017.

\bibitem{Ka90}
Jarkko Kari.
\newblock Reversibility of 2d cellular automata is undecidable.
\newblock {\em Physica D: Nonlinear Phenomena}, 45(1--3):379 -- 385, 1990.

\bibitem{Ka92}
Jarkko Kari.
\newblock The nilpotency problem of one-dimensional cellular automata.
\newblock {\em SIAM J. Comput.}, 21(3):571--586, 1992.

\bibitem{Ka08}
Jarkko Kari.
\newblock {\em SOFSEM 2008: Theory and Practice of Computer Science: 34th
  Conference on Current Trends in Theory and Practice of Computer Science,
  Nov{\'y} Smokovec, Slovakia, January 19-25, 2008. Proceedings}, chapter On
  the Undecidability of the Tiling Problem, pages 74--82.
\newblock Springer Berlin Heidelberg, Berlin, Heidelberg, 2008.

\bibitem{Ka12a}
Jarkko Kari.
\newblock Universal pattern generation by cellular automata.
\newblock {\em Theor. Comput. Sci.}, 429:180--184, 2012.

\bibitem{KaKo17}
Jarkko Kari and Johan Kopra.
\newblock Cellular automata and powers of p/ q.
\newblock {\em RAIRO-Theoretical Informatics and Applications}, 51(4):191--204,
  2017.

\bibitem{KaOl08}
Jarkko Kari and Nicolas Ollinger.
\newblock Periodicity and immortality in reversible computing.
\newblock In {\em Proceedings of the 33rd international symposium on
  Mathematical Foundations of Computer Science}, MFCS '08, pages 419--430,
  Berlin, Heidelberg, 2008. Springer-Verlag.

\bibitem{Kh89}
Olga Kharlampovich.
\newblock The word problem for solvable lie algebras and groups.
\newblock {\em Matematicheskii Sbornik}, 180(8):1033--1066, 1989.

\bibitem{KiRo90}
K.~H. Kim and F.~W. Roush.
\newblock On the automorphism groups of subshifts.
\newblock {\em Pure Mathematics and Applications}, 1(4):203--230, 1990.

\bibitem{Ko19}
Johan Kopra.
\newblock The {L}yapunov exponents of reversible cellular automata are
  uncomputable.
\newblock In {\em International Conference on Unconventional Computation and
  Natural Computation}, pages 178--190. Springer, 2019.

\bibitem{Ku97}
Petr K{\r{u}}rka.
\newblock On topological dynamics of {T}uring machines.
\newblock {\em Theoret. Comput. Sci.}, 174(1-2):203--216, 1997.

\bibitem{Li87}
Douglas Lind.
\newblock Entropies of automorphisms of a topological markov shift.
\newblock {\em Proceedings of the American Mathematical Society},
  99(3):589--595, 1987.

\bibitem{LiMa95}
Douglas Lind and Brian Marcus.
\newblock {\em An introduction to symbolic dynamics and coding}.
\newblock Cambridge University Press, Cambridge, 1995.

\bibitem{Lu10a}
Ville Lukkarila.
\newblock Sensitivity and topological mixing are undecidable for reversible
  one-dimensional cellular automata.
\newblock {\em J. Cellular Automata}, 5(3):241--272, 2010.

\bibitem{LySc15}
R.C. Lyndon and P.E. Schupp.
\newblock {\em Combinatorial Group Theory}.
\newblock Classics in Mathematics. Springer Berlin Heidelberg, 2015.

\bibitem{LyMyUs10}
Igor Lysenok, Alexei Myasnikov, and Alexander Ushakov.
\newblock The conjugacy problem in the {G}rigorchuk group is polynomial time
  decidable.
\newblock {\em Groups, Geometry, and Dynamics}, 4(4):813--833, 2010.

\bibitem{MaMyNiVa15}
Jeremy Macdonald, Alexei Myasnikov, Andrey Nikolaev, and Svetla Vassileva.
\newblock Logspace and compressed-word computations in nilpotent groups.
\newblock {\em arXiv preprint arXiv:1503.03888}, 2015.

\bibitem{Mi68a}
K.~A. Mihailova.
\newblock The occurrence problem for free products of groups.
\newblock {\em Mathematics of the USSR-Sbornik}, 4(2):181, 1968.

\bibitem{Mi71}
Charles~F. Miller~III et~al.
\newblock {\em On group-theoretic decision problems and their classification}.
\newblock Princeton university press, 1971.

\bibitem{Mo91}
C.~Moore.
\newblock Generalized shifts: unpredictability and undecidability in dynamical
  systems.
\newblock {\em Nonlinearity}, 4(2):199, 1991.

\bibitem{Mo97}
Lee Mosher.
\newblock Central quotients of biautomatic groups.
\newblock {\em Commentarii Mathematici Helvetici}, 72(1):16--29, 1997.

\bibitem{Ne68}
B.~B. Newman.
\newblock Some results on one-relator groups.
\newblock {\em Bulletin of the American Mathematical Society}, 74(3):568--571,
  1968.

\bibitem{No54}
Petr~Sergeevich Novikov.
\newblock Unsolvability of the conjugacy problem in the theory of groups.
\newblock {\em Izvestiya Rossiiskoi Akademii Nauk. Seriya Matematicheskaya},
  18(6):485--524, 1954.

\bibitem{No55}
Petr~Sergeevich Novikov.
\newblock On the algorithmic unsolvability of the word problem in group theory.
\newblock {\em Trudy Mat. Inst. Steklov}, 44:3--143, 1955.

\bibitem{Ol13}
Jeanette Olli.
\newblock Endomorphisms of sturmian systems and the discrete chair substitution
  tiling system.
\newblock {\em Dynamical Systems}, 33(9):4173--4186, 2013.

\bibitem{Je12}
Jean-Philippe {Pr{\'e}aux}.
\newblock {The conjugacy problem in groups of non-orientable 3-manifolds}.
\newblock {\em arXiv e-prints}, page arXiv:1202.4148, February 2012.

\bibitem{Re69}
Vladimir~N Remeslennikov.
\newblock Conjugacy in polycyclic groups.
\newblock {\em Algebra and Logic}, 8(6):404--411, 1969.

\bibitem{Ry72}
J.~Patrick Ryan.
\newblock The shift and commutativity.
\newblock {\em Mathematical systems theory}, 6(1-2):82--85, 1972.

\bibitem{Ol10}
Olga~Patricia Salazar-D{\'\i}az.
\newblock Thompson's group {V} from a dynamical viewpoint.
\newblock {\em International Journal of Algebra and Computation},
  20(01):39--70, 2010.

\bibitem{Sa12b}
Ville Salo.
\newblock A characterization of cellular automata generated by idempotents on
  the full shift.
\newblock In Edward~A. Hirsch, Juhani Karhum{\"a}ki, Arto Lepist{\"o}, and
  Michail Prilutskii, editors, {\em CSR}, volume 7353 of {\em Lecture Notes in
  Computer Science}, pages 290--301. Springer, 2012.

\bibitem{Sa13}
Ville {Salo}.
\newblock {Hard Asymptotic Sets for One-Dimensional Cellular Automata}.
\newblock {\em arXiv e-prints}, page arXiv:1307.4910, July 2013.

\bibitem{Sa14d}
Ville Salo.
\newblock {Toeplitz subshift whose automorphism group is not finitely
  generated}.
\newblock {\em Colloquium Mathematicum}, 146:53--76, 2017.

\bibitem{Sa17b}
Ville Salo.
\newblock Transitive action on finite points of a full shift and a finitary
  {R}yan’s theorem.
\newblock {\em Ergodic Theory and Dynamical Systems}, pages 1--31, 2017.

\bibitem{Sa18d}
Ville Salo.
\newblock A note on subgroups of automorphism groups of full shifts.
\newblock {\em Ergodic Theory Dynam. Systems}, 38(4):1588--1600, 2018.

\bibitem{Sa18c}
Ville {Salo}.
\newblock {Universal gates with wires in a row}.
\newblock {\em ArXiv e-prints}, September 2018.
\newblock Available at \url{https://arxiv.org/abs/1604.01646}. Accepted in
  Journal of Algebraic Combinatorics.

\bibitem{Sa18a}
Ville {Salo}.
\newblock {Universal groups of cellular automata}.
\newblock {\em ArXiv e-prints}, August 2018.
\newblock Available at \url{https://arxiv.org/abs/1808.08697}. To appear in
  Colloquium Mathematicum.

\bibitem{Sa19a}
Ville Salo.
\newblock No {T}its alternative for cellular automata.
\newblock {\em {Groups, Geometry and Dynamics}}, 13:1437–1455, 2019.

\bibitem{Sa20a}
Ville {Salo}.
\newblock {Universal CA groups with few generators}.
\newblock {\em arXiv e-prints}, page arXiv:2002.12713, February 2020.

\bibitem{SaTo15d}
Ville Salo and Ilkka T\"orm\"a.
\newblock Block maps between primitive uniform and pisot substitutions.
\newblock {\em Ergodic Theory and Dynamical Systems}, 35:2292--2310, 10 2015.

\bibitem{Sc08}
Saul Schleimer.
\newblock Polynomial-time word problems.
\newblock {\em Commentarii mathematici helvetici}, 83(4):741--765, 2008.

\bibitem{SuVe12}
Zoran \u{S}uni\'{c} and Enric Ventura.
\newblock The conjugacy problem in automaton groups is not solvable.
\newblock {\em Journal of Algebra}, 364:148 -- 154, 2012.

\bibitem{We19}
Linda {Westrick}.
\newblock {Topological completely positive entropy is no simpler in $\mathbb
  Z^2$-SFTs}.
\newblock {\em arXiv e-prints}, page arXiv:1904.11444, April 2019.

\end{thebibliography}

\end{document}